\documentclass[11pt]{amsart}
\usepackage[utf8]{inputenc}  
\usepackage[T1]{fontenc}
\usepackage[english]{babel}
\usepackage{amsmath,textcomp,amssymb,geometry,graphicx,enumerate,multicol}
\usepackage{multirow}
\usepackage{subcaption}
\usepackage{algorithm} 
\usepackage[noend]{algpseudocode} 
\usepackage{hyperref}
\usepackage{stmaryrd}
\usepackage[foot]{amsaddr}

\newtheorem{theorem}{Theorem}[section]
\newtheorem{lemma}[theorem]{Lemma}
\newtheorem{proposition}[theorem]{Proposition}
\newtheorem{corollary}[theorem]{Corollary}

  \newtheorem{remark}[theorem]{Remark}
  \newtheorem{definition}[theorem]{Definition}
  \newtheorem{example}[theorem]{Example}
\numberwithin{equation}{section}

\hypersetup{
    colorlinks=true,
    linkcolor=blue,
    filecolor=magenta,      
    urlcolor=blue,
}

\newcommand\qbin[3]{\left[\begin{matrix} #1 \\ #2 \end{matrix} \right]_{#3}}

\DeclareMathOperator{\con}{con}

\title[Naruse hook formula for mobile posets]{Naruse hook formula for linear\\ extensions of mobile posets}
\author{GaYee Park}
\address{Department of Mathematcs and Statistic, University of Massachusetts, Amherst, MA}
\email{park@math.umass.edu}
\date{\today}

\usepackage[utf8]{inputenc}  
\usepackage[T1]{fontenc}
\usepackage[english]{babel}
\usepackage{amsmath,textcomp,amssymb,geometry,graphicx,enumerate,multicol}
\usepackage{multirow}
\usepackage{algorithm} 
\usepackage[noend]{algpseudocode} 
\usepackage{hyperref}
\usepackage{stmaryrd}

\hypersetup{
    colorlinks=true,
    linkcolor=blue,
    filecolor=magenta,      
    urlcolor=blue,
}

\newcommand{\Cat}{\mathrm{Cat}}
\newcommand{\SYT}{\mathrm{SYT}}
\newcommand{\inv}{\mathrm{inv}}
\newcommand{\maj}{\mathrm{maj}}
\newcommand{\des}{\mathrm{Des}}
\newcommand{\E}{\mathcal{E}}
\newcommand{\stat}{\mathrm{stat}}
\newcommand{\Br}{\mathrm{Br}}

\begin{document}

\maketitle

\begin{abstract}
   Linear extensions of posets are important objects in enumerative and algebraic combinatorics that are difficult to count in general. Families of posets like Young diagrams of straight shapes and $d$-complete posets have hook-length product formulas to count linear extensions, whereas families like Young diagrams of skew shapes have determinant or positive sum formulas like the Naruse hook-length formula from 2014. In 2020, Garver et. al. gave determinant formulas to count linear extensions of a family of posets called mobile posets that refine $d$-complete posets and border strip skew shapes. We give a Naruse type hook-length formula to count linear extensions of such posets by proving a major index $q$-analogue. We also a inversion index $q$-analogue of the Naruse formula for mobile tree posets. 
 
\end{abstract}

\section{Introduction}

\subsection{Hook-length formulas for linear extensions}
Linear extensions of posets are fundamental objects in combinatorics. In general, computing the number $e(P)$ linear extensions of any poset is a difficult problem, it is $\#P$-complete \cite{Brightwell-Winkler}. For certain posets like Young diagrams, rooted trees, and more generally $d$-complete posets, there are product formulas that compute the number of linear extensions efficiently, such as the classical \emph{hook-length formula} (HLF) for the number of standard Young tableaux (SYT) of shape $\lambda$. 

\begin{theorem}[{Frame-Robinson-Thrall \cite{FRT}}]
Let $\lambda$ be a partition of $n$. We have
    \begin{equation} \tag{HLF} \label{eq:HLF}
       |\SYT(\lambda)| =  n!\prod_{u \in [\lambda]}\frac{1}{h(u)}
    \end{equation}
where $h(u) = \lambda_i + \lambda_j' -i - j +1$ is the hook-length of the square $u= (i,j)$.
\end{theorem}

For skew shapes, there is no known product formula, however, Naruse introduced a generalization of the hook-length formula as a positive sums over excited diagrams of products of hook-lengths. We call this the \emph{Naruse hook-length formula} (NHLF).

\begin{theorem}[{Naruse \cite{naruse2014}}]
For a skew shape $\lambda/\mu$ of size $n$, we have
\begin{equation} \label{eq:NHLF}
    |\SYT(\lambda/\mu)| \,=\, n!\sum_{D \in \mathcal{E}(\lambda/\mu)} \prod_{u\in [\lambda] \setminus D}\frac{1}{h(u)}, \tag{NHLF}
\end{equation}
where $\E(\lambda/\mu)$ is the set of excited diagrams of $\lambda/\mu$.
\end{theorem}

The number of SYT of shape $\lambda/\mu$ can also be interpreted as the number of linear extension of a poset induced by the Young diagram of $\lambda/\mu$. In \cite{Pro}, Proctor  defined the family of \emph{d-complete} posets, that include Young diagrams of shape $\lambda$ and rooted trees and have a hook-length formula to count the number of linear extensions.

\begin{theorem}[{Peterson-Proctor \cite{Pro}}]\label{thm: proctor}
The number of linear extensions of a $d$-complete poset $P$ with $n$ element is
\[e(P) = \frac{n!}{\prod_{x\in P}h_{P}(x)},\]
where $h_{P}(x)$ is the hook-length of $x\in P$.
\end{theorem}

\subsection{$q$-analogue of hook-length formulas}
There are the following $q$-analogues of both the HLF and the NHLF for semistandard Young tableaux. We state these results in terms of $e_q^{\maj}(P,\omega):=\sum_{\sigma} q^{\maj(\sigma)}$ where $(P,\omega)$ is a labeled poset, and $\sigma$ is a linear extension of it. This polynomial also encodes the generating functions of $P$-partitions \cite{Sta1}.

\begin{theorem}[Stanley \cite{Sta3}]
For a shape $\lambda$ with associated poset $Q_\lambda$ of size $n$ with $\omega$ Schur labeling, we have:
\begin{equation}
   \frac{e_q^{\maj}(Q_{\lambda},\omega)}{\prod^n_{i=1}(1-q^i)}  =  q^{b(\lambda)} \prod_{u \in [\lambda]} \frac{1}{1-q^{h(u)}},
\end{equation}
where $b(\lambda) = \sum_{i}(i-1)\lambda_i$.
\end{theorem}

\begin{theorem}[{Morales-Pak-Panova \cite{MPP1}}]\label{thm: MPP1 maj}
For a skew shape $\lambda/\mu$ with associated poset $Q_{\lambda/\mu}$ of size $n$ with $\omega$ Schur labeling, we have:
\begin{equation}
   \frac{e_q^{\maj}(Q_{\lambda/\mu},\omega)}{\prod^n_{i=1}(1-q^i)}  = \sum_{D\in \mathcal{E}(\lambda/\mu)} q^{w(D)} \prod_{u \in [\lambda]\setminus D} \frac{1}{1-q^{h(u)}},
\end{equation}
where $w(D) = \sum_{u\in Br(D)} h(u)$ is the sum of hook-lengths of the support of broken diagonals.
\end{theorem}

For $d$-complete posets, we also have the following $q$-analogue in terms of major index (see Section~\ref{subsec: qanalogue of linear extensions}).

\begin{theorem}[{Peterson and Proctor \cite{Pro}}]\label{thm: d-complete major index}
For a labeled $d$-complete poset $(P, \omega)$ of size $n$ with $\omega$ any labeling, we have:
\[e_q^{\maj}(P,\omega) = q^{\maj(P,\omega)} \frac{[n]_q!}{\prod_{x\in P} [h_{P}(x)]_q},\]
where $h_{P}(x)$ is the hook-length of $x\in P$ and $\maj(P, \omega) = \sum_{x\in \des(P,\omega)}h_{P}(x)$.
\end{theorem}

\subsection{Hook formulas for mobile posets}
A \emph{border strips} is a connected skew-shaped diagram with no $2\times 2$ box.
A \emph{mobile poset} is a recent common refinement of border strips and $d$-complete posets introduced in \cite{GMM} (see Figure~\ref{fig: poset to mobile}, (a)). The authors found a determinantal formula for the number of linear extensions of these posets, similar to \emph{Jacobi--Trudi} formula and asked whether there was a Naruse-type formula \cite[Sec. 6.1]{GMM} for this number. The first main result of this paper is a $q$-analogue Naruse hook-length formula for mobile posets for the major index, generalizing Theorem~\ref{thm: MPP1 maj}.

\begin{theorem}\label{thm: major mobile nhlf}
Let $P_{\lambda/\mu}(\mathbf{p})$ be a free-standing mobile poset of size $n$ with with underlying border strip $\lambda/\mu$ and $\mathbf{p} = (p_{(r_1,s_1)},\dots, p_{(r_k,s_k)})$ the $d$-complete posets hanging on $(r,s)$. For a labeled mobile poset $(P_{\lambda/\mu}(\mathbf{p}),\omega)$ with $\omega$ reversed Schur labeling on $[\lambda/\mu]$ and natural labeling on $d$-complete posets, we have:
    \[\frac{e_q^{\maj}(P_{\lambda/\mu},\omega)}{\prod^n_{i=1}(1-q^i)} = \prod_{v\in \mathbf{p}}\frac{1}{1-q^{h(v)}}\sum_{D\in \mathcal{E}(\lambda/\mu)} q^{w'(D)} \prod_{u\in [\lambda]\setminus D} \frac{1}{1-q^{h'(u)}},\] 
     where $h(v)$ is the hook-length of the element $v$ in the $d$-complete posets in $\mathbf{p}$, $h'(i,j) =  \lambda_i - i + \lambda'_j-j + 1 + \sum_{r\geq i,s\geq j}|p_{({r,s})}|$ and $w'(D) = \sum_{u\in \Br(D)} h'(u)$ is the sum of hook-lengths of the supports of broken diagonals.
\end{theorem}

This result is a common refinement of Theorem~\ref{thm: MPP1 maj} and Theorem~\ref{thm: d-complete major index}. Note that Naruse-Okada \cite{NaruseOkada} have a different $q$-analogue of $e_q^{\maj}(P,\omega)$ for a family called \emph{skew $d$-complete} posets with \emph{natural labelings}. See Section~\ref{subsec: skew d-complete}.  

The proof of Theorem~\ref{thm: major mobile nhlf} is based on the method used in \cite{MPP2} to prove \eqref{eq:NHLF} for border strips. This involves the \emph{Pieri--Chevalley formula} \eqref{eq: Chevalley} and a recurrence of linear extensions for mobile posets. The proof of the latter is combinatorial and uses a generalization of Stanley's theory of $(P,\omega)$-partitions \cite{Sta1}. 

By taking $q=1$, we obtain a Naruse hook-length formula for mobile posets as a corollary.

\begin{corollary}[NHLF for mobiles]\label{thm: mobile nhlf}
For a free-standing mobile poset $P_{\lambda/\mu}(\mathbf{p})$ of size $n$, we have:
\begin{equation} \label{eq: main thm}
    e(P_{\lambda/\mu}({\bf p})) = \frac{n!}{H({\bf p})} \sum_{D\in \E(\lambda/\mu)} \prod_{(i,j)\in \gamma} \frac{1}{h'(i,j)},
\end{equation}
where $H(\mathbf{p})$ is the product of hook-lengths of all elements in the $d$-complete posets hanging from $\lambda/\mu$.
\end{corollary}

As an application of this corollary, we give bounds to  generalizations of Euler number defined in \cite{GMM}. See Corollary~\ref{cor: boundary} and Corollary~\ref{cor: zigzag bound}.

Our second result is a $q$-analogue of NHLF for mobile tree posets (the hanging $d$-complete posets are restricted to rooted trees) in terms of the inversion statistic, $e_q^{\inv}(P,\omega): = \sum_{\sigma} q^{\inv(\sigma)}$, where $\sigma$ is a linear extension of $(P,\omega)$.

\begin{theorem}\label{thm: inv nhlf}
For a labeled mobile poset $(P_{\lambda/\mu},\omega)$ with $\omega$ reversed Schur labeling on $[\lambda/\mu]$ and natural labeling on $d$-complete posets,
    \begin{equation}\label{thm: inv q-analogue}
   \frac{e_q^{\inv}(P_{\lambda/\mu},\omega)}{\prod^n_{i=1}(1-q^i)} = \prod_{v\in \mathbf{p}}\frac{1}{1-q^{h(v)}}\sum_{D\in \mathcal{E}(\lambda/\mu)} q^{w(D)+p_D} \prod_{u\in [\lambda]\setminus D} \frac{1}{1-q^{h'(u)}}, \end{equation}
     where $w(D) = \sum_{u\in Br(D)} h(u)$ is the sum of hook-lengths of the supports of broken diagonals and $p_D = \sum_{(i,j)\in [\mu]\setminus D}\sum_{b=j} p_{r,s}$.
\end{theorem}

\subsection{Paper outline}
In Section~\ref{sec: background}, we give definitions and background results required for the proof. In Section~\ref{sec: p-partition}, we give results for $P$-partition with a fixed point. In Section~\ref{sec: q-analogue major}, we give an example and the proof of Theorem~\ref{thm: major mobile nhlf}. In Section~\ref{sec: application}, we show an application to the Corollary~\ref{thm: mobile nhlf}. In Section~\ref{sec: q-analogue inversion}, we give examples and the proof of the inversion index of the case of $q$-analogue. Lastly, we end with final remarks in Section~\ref{sec: final remarks}.

\section{Background and Preliminaries} \label{sec: background}

\subsection{Posets and linear extensions}

A \emph{partially-ordered set} (poset) is a pair $(P,\leq_P)$ where $P$ is a finite set and $\leq_P$ is a binary relation that is reflexive, anti-symmetric, and transitive. A \emph{linear extension} of an $n$-element poset $P$ is a bijection $f: P \to [n]$ that is order-preserving. The number of SYT of a shape $\lambda/\mu$ is equal to the number of \emph{linear extensions} of a poset of shape $\lambda/\mu$. We denote  the set of linear extensions of $P$ as $\mathcal{L}(P)$, and $e(P) = |\mathcal{L}(P)|$.  


\subsection{Border strips and Mobile Posets}
A \emph{border strip} is a connected skew shape $\lambda/\mu$ containing no $2\times 2$ box. \emph{d-complete posets} are a large class of posets containing rooted tree posets and posets arising from Young diagrams. Given a border strip, we can convert it into a poset by letting the inner corners of the diagram be the maximal points of the corresponding poset. Now we can construct a mobile poset (See Figure~\ref{fig: poset to mobile} for an example).

\begin{definition}[{Garver-Grosser-Matherne-Morales \cite{GMM}}]
A \emph{mobile\footnote{What we call a mobile poset is called a free-standing mobile poset in \cite{GMM}.} (tree) poset} is a poset obtained from a border strip $Q$, by allowing every element $x \in Q$ to cover the maximal element of a nonnegative number of disjoint  $d$-complete (rooted tree) posets.
\end{definition}

\begin{figure}[h]
    \centering
    \includegraphics[width=4.5cm]{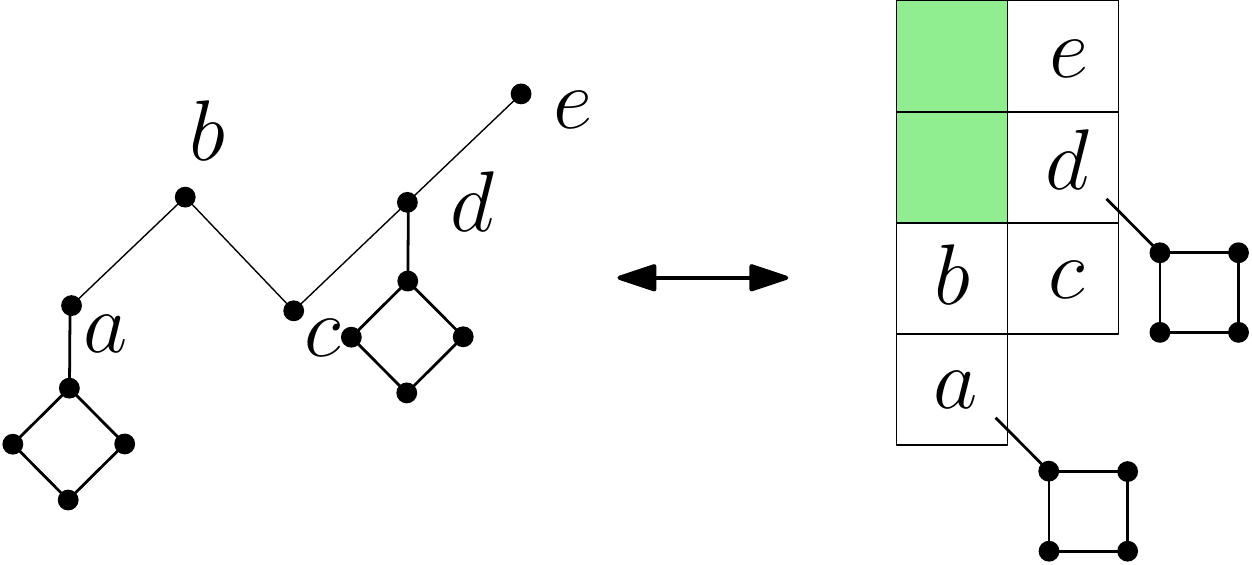}
    \caption{The conversion of a mobile poset to a young diagram with $d$-complete posets attached.}
    \label{fig: poset to mobile}
\end{figure}

\subsection{Excited diagrams and broken diagonals}

Denote $[\lambda/\mu]$ as the skew shape Young diagram of a shape $\lambda/\mu$. An \emph{excited diagram} of $\lambda/\mu$, denoted by $D$, is a subset of $[\lambda]$ obtained from $\mu$ by applying a sequence of \emph{excited moves} that we define next. Let $D\in \E(\lambda/\mu)$, then $(i,j) \in D$ is an \emph{active cell} if $(i+1,j), (i,j+1)$, and $(i+1,j+1)$ are not in $D$. We obtain a new excited diagram by replacing an active cell by $(i+1,i+j)$ (see Figure~\ref{fig: poset to mobile}, (b)). Note that for border strips, the excited diagrams can also be interpreted as the complement of its lattice paths $\gamma$ from $(\lambda_1',1)\to (1,\lambda_1)$ that stay inside $[\lambda]$. (see\cite[Sec. 3]{MPP2}).

For each excited diagram $D\in \E(\lambda/\mu)$ we associate a set of \emph{broken diagonals} $Br(D) \subset [\lambda]\setminus D$ as follows. Start with $D = [\mu]$, then $\Br(D) = \{(i,j)\in\lambda/\mu | i-j = \mu_t-t\}$, where $\mu_t=0$ if $\ell(\mu)<t\leq \ell(\lambda)$. For each active cell $u=(i,j)$ and its excited move $\alpha_u:D\to D'$, we have a corresponding move for the broken diagonal where $\Br(D') = \Br(D)\setminus \{(i+1,j+1)\} \cup \{ (i+1,j)\}\}$. See Figure \ref{fig: excited moves}, (a) and (b).

\begin{figure}[h]
     \begin{subfigure}[normal]{0.3\textwidth}
     \centering
    \includegraphics[height=.8cm]{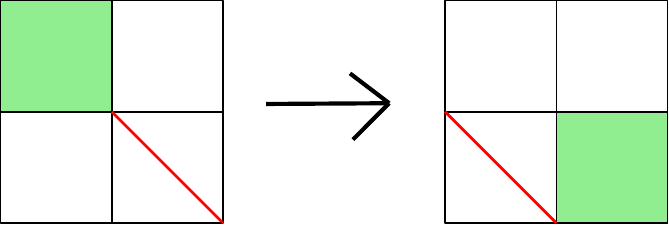}
    \caption{}
    \end{subfigure}
    \quad
    \begin{subfigure}[normal]{0.3\textwidth}
    \centering
    \includegraphics[height=1.5cm]{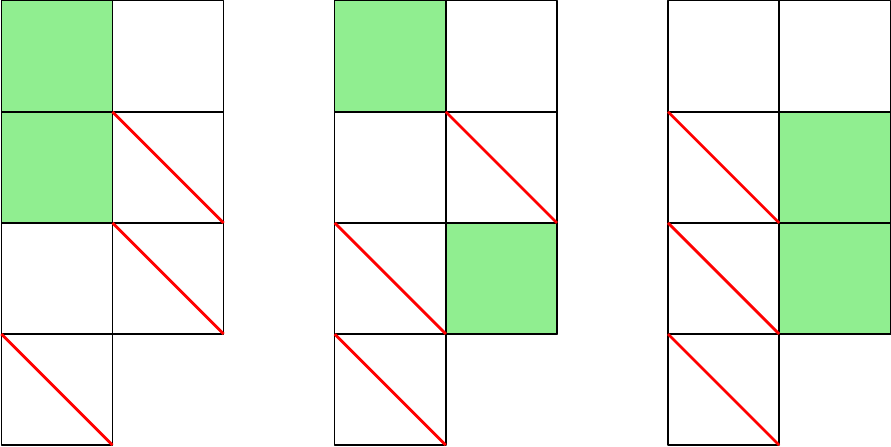}
    \caption{}
    \end{subfigure}
    \caption{(a) Excited move (b) Excited diagrams and the corresponding broken diagonals of $\lambda/\mu =(2,2,2,1)/(1,1)$}
    \label{fig: excited moves}
\end{figure}

\subsection{Multivariate function}

For border strip $\lambda/\mu$, let

\begin{equation}{\label{eq: multivariate}}
F_{\lambda/\mu}(\mathbf{x}, \mathbf{y})= F_{\lambda/\mu}(x_1,\dots, x_{\lambda_1'}, y_1,\dots, y_{\lambda_1}) : = \sum_{D\in \E(\lambda/\mu)} \prod_{(i,j)\in [\lambda]\setminus D} \frac{1}{x_i - y_j}.
\end{equation}

Let $\lambda/\nu$ be the shape obtained by removing an inner corner of $\lambda/\mu$. We denote this subtraction by $\mu\to\nu$. Note that $[\lambda/\nu]$ is a disconnected skew shaped.

We need the following identity of $F_{\lambda/\mu}(\mathbf{x}, \mathbf{y}) $ from \cite{MPP2}.

\begin{lemma}[{Pieri--Chevalley formula \cite[Eq. (6.3)]{MPP2}}] 
\begin{align}\label{eq: Chevalley}
F_{\lambda/\mu} (\mathbf{x},\mathbf{y}) &= \frac{1}{x_1-y_1}\sum_{\mu \to \nu} F_{\lambda/\nu^1}({\bf x},{\bf y}) F_{\lambda/\nu^2}({\bf x},{\bf y}), 
\end{align}
where $\lambda/\nu^1$ and $\lambda/\nu^2$ are the two connected border strips that form $\lambda/\nu$.
\end{lemma}


\subsection{$q$-analogues of linear extensions}\label{subsec: qanalogue of linear extensions}

A \emph{labeled poset} $(P,\omega)$ is a poset $P$ with a labeling $\omega:P \to [n]$. We call $\omega$ a \emph{natural labeling} if for any $x,y\in P$ with $x<_{P} y$, we have $\omega(x) < \omega(y)$ \cite{Sta1}. We call $\omega$ a \emph{reversed Schur labeling} if the labeling decreases as it follows the path from the bottom left to the top right of the tableau. For Theorem~\ref{thm: inv nhlf} and Theorem~\ref{thm: major mobile nhlf}, we use reversed Schur labeling on $[\lambda/\mu]$ and natural labeling on the $d$-complete posets. In the case of the inversion statistic, for each $\mu\to\nu$, we need $\omega(x_1)>\omega(x_2)$ for all $x_i \in \lambda/\nu_i$ to satisfy the condition for Proposition~\ref{prop: qanalogue disjoint poset inv}

Given a linear extension $f: P \to [n]$, the permutation $\omega \circ f^{-1} \in \mathfrak{S}_n$ is a \emph{linear extension of the labeled poset}, $\mathcal{L}(P,\omega)$. For the major index, we label the poset using the \emph{natural labeling} on the $d$-complete posets and the \emph{reversed Schur labeling} on the border-strip, where the labeling is increasing from right to left and increasing from top to bottom of the diagram (See Figure~\ref{fig:mobile example} (b)). Such labeling is derived from the Schur labeling of $\lambda/\mu$. We use the reversed Schur labeling of instead of the Schur labeling because of the orientation of the conversion from a Young diagram to a poset we have chosen (see Figure~\ref{fig: poset to mobile}).

Recall, $\inv(\sigma)= \#\{i\in [n-1]| \sigma_i > \sigma_j \mbox{ where } i<j\}$ where $\des(\sigma):= \{i \in [n-1] \,|\, \sigma_i > \sigma_{i+1}\}$ and $\maj(\sigma)= \sum_{i\in \des(\sigma)} i$. The two common statistics for $q$-analogues of the number of linear extensions for a labeled poset $(P,\omega)$ are the major index and inversions. There is a version of the major index and the inversion index for labeled $d$-complete posets.

\begin{definition}
The {major index} of a labeled $d$-complete poset $(P,\omega)$ is
\[\maj(P,\omega) = \sum_{x\in \des(P,\omega)}h_{P}(x).\]
\end{definition}

\begin{definition}
The {inversion index} of a labeled poset $(P,\omega)$ is
\[\inv(P,\omega) = \#\{(x,y) \,|\, \omega(y) < \omega(x) \mbox{ and } x<_{P} y\}.\]
\end{definition}

Let  $\mbox{stat}\in \{\maj, \inv\}$, the \emph{major index (inversion) q-analogue} of the number of linear extensions of a labeled poset $(P,\omega)$ is 
\[e_q^{\stat}(P,\omega) := \sum_{\sigma\in \mathcal{L}(P,\omega)}q^{\stat(\sigma)}.
\]

Now we state the $q$-analogue of the hook-length formulas. We define the \emph{q-integer} $[n]_q:= 1+q+\cdots + q^{n-1}$ and the \emph{q-factorial} $[n]_q! := [n]_q\cdots[2]_q[1]_q$.
Peterson and Proctor \cite{Pro} gave a $\maj$ $q$-analogue formula for $d$-complete posets.

\begin{theorem}[{Peterson and Proctor \cite{Pro}}]\label{thm: maj d-complete poset}
Let $(P, \omega)$ be a labeled $d$-complete poset of size $n$ with any labeling. Then,
\[e_q^{\maj}(P,\omega) = q^{\maj(P,\omega)} \frac{[n]_q!}{\prod_{x\in P} [h_{P}(x)]_q}.\]
\end{theorem}

Bj\"{o}rner and Wachs \cite{BW} gave a $\inv$ $q$-analogue formula for rooted tree posets.

\begin{theorem}[{Bj\"{o}rner and Wachs \cite{BW}}]\label{thm: inv tree poset}
Let $(P, \omega)$ be a rooted tree poset with a natural labeling. Then
\[e_q^{\inv}(P,\omega) = q^{\inv(P,\omega)} \frac{[n]_q!}{\prod_{x\in P} [h_{P}(x)]_q}.\]
\end{theorem}

\subsection{Identities for $e_q^{\maj}$ from the theory of $P$-partitions}

In this section, we state the definition of $(P,\omega)$-partition and its connection to $e_q^{\maj}(P,\omega)$. 
\begin{definition}\cite{Sta1}
A $(P,\omega)$-partition is a map $f: P \to \mathbb{N}$ satisfying the conditions:
\begin{enumerate}
    \item If $s\leq t$ in P, then $f(s) \geq$ f(t).
    \item If $s\leq t$ and $\omega(s) > \omega(t)$, then $f(s) > f(t)$.
\end{enumerate}
\end{definition}
If $\sum_{t\in P} f(t) = n$, then we say $f$ is a $(P,\omega)$-partition of $n$. We denote the set of all $(P,\omega)$-partitions as $\mathcal{A}(P,\omega)$.
First, we have the following definition.

\begin{definition}
Let $w = w_1w_2\dots w_n \in \mathfrak{S}_n$. We say that the function $f':[n] \to \mathbb{N}$ is $w$-compatible if the following two conditions hold.
\begin{enumerate}
    \item $f'(w_1)\geq f'(w_2) \geq \cdots \geq f'(w_n)$.
    \item $f'(w_i) > f'(w_{i+1})$ if $w_i > w_{i+1}$.
\end{enumerate}
\end{definition}

Define $f': [p] \to \mathbb{N}$ by $f'(i) = f(\omega^{-1}(i))$, and let $S_w$ be the set of all functions $f$ such that $f'$ is $w$-compatible. We have the following result called the fundamental lemma on $(P,\omega)$-partitions.

\begin{lemma}[\cite{Sta1}, Lemma 3.15.3]\label{lemma: fundamental p-partition}
A function $f:P \to \mathbb{N}$ is a $(P,\omega)$-partition if and only if $f'$ is $w$-compatible with some $w\in \mathcal{L}(P,\omega)$. In other words,
\[\mathcal{A}(P,\omega) = \dot\bigcup_{w\in \mathcal{L}(P,\omega)}S_w,\] where $\dot\cup$ denotes disjoint union.
\end{lemma}

 Let $a_P(n)$ denote the number of $(P,\omega)$-partitions of $n$ and let $G_{P,\omega}(x) = \sum a_P(n)x^n$ be generating function of these partitions. Stanley gave the following specialization of the generating function associated with $(P,\omega)$-partition:

\begin{theorem}[Stanley \cite{Sta1}]\label{thm: G(P)}
Let $(P,\omega)$ be a labeled poset of size $p$. Then we have,
\[G_{P,\omega}(q) = \frac{e_q^{\maj}(P,\omega)}{\prod_{i=1}^p(1-q^i)}. \]
\end{theorem}

For any disjoint union of labeled posets $(P,\omega_1) + (Q,\omega_2)$, by definition, we have since \[G_{P+Q,\omega_1+\omega_2}(q)= G_{P,\omega_1}(q) \cdot G_{Q,\omega_2}(q).\] 
By applying Theorem~\ref{thm: G(P)} to above equation, we have the following corollary.

\begin{corollary}\cite[Exercise 3.162(a)]{Sta1}\label{cor: qanalogue disjoint maj}
Let $(P+Q,\omega)$ be a labeled disjoint sum of posets with $|P+Q| = n$ and $|P|=n$. For any labeling $\omega$, we have
\[e_q^{\maj}(P +Q,\omega) = \qbin{n}{p}{q} e_q^{\maj}(P,\omega_1)\cdot e_q^{\maj}(Q,\omega_2),\]
where $\omega_1$ and $\omega_2$ is the labeling obtained by restricting $\omega$ to $P$ and $Q$ respectively.
\end{corollary}

\subsection{Identities for $e_q^{\inv}$ for disjoint union of posets}

We also have the following disjoint union of poset identity for the case of inversion index.

\begin{proposition}[Bj\"{o}rner-Wachs,\cite{BW}]\label{prop: qanalogue disjoint poset inv}

Let $(P+Q,\omega)$ be a labeled disjoint sum of posets with $|P+Q| = n$ and $|P| = p$. Suppose that $\omega$ has the property that the label of every element of $P$ is smaller than the label of every element of $Q$. We have
\[e_q^{\inv}(P+Q,\omega) = \qbin{n}{p}{q} e_q^{\inv}(P,\omega_1)\cdot e_q^{\inv}(Q,\omega_2),\]
where $\omega_1$ and $\omega_2$ are the labeling obtained by restricting $\omega$ to $P$ and $Q$ respectively.
\end{proposition}

\begin{remark}
Note that the disjoint poset identity for the inversion statistic has a more specific condition on the poset labeling than the major index does. When applying the Theorem~\ref{thm: inv nhlf}, we need to label the mobile posets so that it satisfies the condition of Proposition~\ref{prop: qanalogue disjoint poset inv}. More detail about the labeling of the poset for the case of inversion index is stated in Section~\ref{sec: q-analogue inversion}.
\end{remark}

\section{$P$-partition with a fixed point}\label{sec: p-partition}

In this section, we discuss a variation of Theorem~\ref{thm: G(P)}, where we fix the position of an element in $P$. As a result, we have Corollary~\ref{cor: e_q identity from G product}, which is used to prove the recurrence in $e_q^{\maj}$.

Denote $\mathcal{L}(P,\omega;s)$ as the set of linear extensions of $(P,\omega)$ that end with $\omega(s)$ for a fixed $s \in P$ and $e_q^{\maj}(P,\omega;s) = \sum_{\sigma\in \mathcal{L}(P,\omega;s)} q^{\maj(\sigma)}$. We omit $\omega$ and assume that the poset is labeled $(P,\omega)$ unless specified. There is analogue of Corollary~\ref{cor: qanalogue disjoint maj} for $e_q^{\maj}(P,\omega;s)$.

We can restrict the conditions on $\mathcal{A}(P,\omega)$ so that it only considers the linear extensions that ends with a fixed element $s\in P$. We define the following restricted $(P,\omega)$-partition:

\begin{definition}
A $(P,\omega;s)$-partition is a map $f:P\to \mathbb{N}$ satisfying the following conditions:
\begin{enumerate}
    \item $f$ is a $(P,\omega)$-partition
    \item $f(s) \leq f(t)$ for all $t \in P$
    \item  $f(s)=f(t)$ then $\omega(s)>\omega(t)$
\end{enumerate}
\end{definition}

Let $\mathcal{A}(P,\omega;s)$ denote the set of such partitions and let 
\[
\mathcal{L}(P,\omega; \{s\}): = \{\sigma \in \mathcal{L}(P,\omega)| \sigma \mbox{ ends with }\omega(s) \}.
\]

\begin{lemma}\label{lemma: fundamental p-partition restricted}
A function $f:P \to \mathbb{N}$ is a $(P,\omega;s)$-partition if and only if $f'$ is $w$-compatible with some $w\in \mathcal{L}(P,\omega)$. In other words,
\[\mathcal{A}(P,\omega;s) = \dot\bigcup_{w\in \mathcal{L}(P,\omega;s)}S_w.\]

\end{lemma}

\begin{proof}
By the definition of a $(P,\omega;s)$-partition, the set $\mathcal{A}(P,\omega;s)$ consists of the $(P,\omega)$-partitions such that $f'$ is $w$-compatible with some $w \in \mathcal{L}(P,\omega; s)$. Then we can restrict the Lemma~\ref{lemma: fundamental p-partition} to $\mathcal{L}(P,\omega;s)$.
\end{proof}


Let $a_P^s(n)$ be the number of $(P,\omega;s)$-partition of size $n$, and let $G_{P;s}(x) = \sum a_P^s(n)x^n$ be the generating function. Then we can restrict Theorem~\ref{thm: G(P)} to $\mathcal{L}(P,\omega;s)$. Denote $e_q^{\maj}(P,\omega;\{s\}):=\sum_{\sigma \in \mathcal{L}(P_{\lambda/\mu},\omega; \{s\})} q^{\maj(\sigma)}$. 

\begin{theorem}\label{thm: G(P) restricted}
Let $(P,\omega)$ be a labeled poset of size $p$ and $s\in P$. Then we have,
\begin{equation*}
    G_{P,\omega;s}(q) = \frac{e_q^{\maj}(P,\omega;s)}{\prod_{i=1}^p(1-q^i)}.
\end{equation*}
\end{theorem}

\begin{proof}
This is a consequence of Lemma~\ref{lemma: fundamental p-partition restricted}.
\end{proof}

We have the following lemma for the disjoint union of posets with the restriction.

\begin{lemma}\label{lemma: GH product restricted}
Let $(P+Q,\omega)$ be a labeled disjoint union of posets and fix $s\in P$ such that $\omega(s) > \omega(t)$ for all $t\in Q$. Let $p = |P|$ and $n = |P+Q|$. Then we have
\begin{equation*}\label{eq: G product}
(1-q^n)\cdot G_{P+Q;s}(q)  = (1-q^p)\cdot G_{P;s}(q)\cdot G_Q(q).    
\end{equation*}

\end{lemma}

\begin{proof}
To prove the lemma, we first give a combinatorial interpretation of $H_{P;s}(q):=(1-q^p)\cdot G_{P;s}(q)$. Consider the difference $G_{P;s}-q^p G_{P;s}$. We build an injection from the $(P;s)$-partition of size $m-p$ to the $(P;s)$-partition of size $m$. The coefficients of $q^m$ from the generating function $q^p \cdot G_{P;s}(q)$ counts the number of $(P;s)$-partitions of size $m-p$. Note that for any $(P;s)$-partition $f$ of size $m-p$, we can obtain a $P$-partition of size $m$ by adding an element for each part of $f$. Likewise, given any $(P;s)$-partition $g$ of size $m$ such that $g(i)>0$ for all $i$, we can obtain a $(P;s)$-partition of size $m-p$ by subtracting one from each part of $g$. Then the coefficients of $q^m$ of $H_{P;s}(q)$ counts the number of $(P;s)$-partitions $g$  of size $m$ with $g(s)$ is a minimum value, $\omega(s)>\omega(t)$ for all $t\in P$ where $g(s)=g(t)$, and with at least one zero value. Thus it is necessary and sufficient for $g(s)=0$. Similarly, $H_{P+Q;s}(q)$ is a generating function for $(P+Q;s)$-partition $f$ of size $n$ such that (i) $f(s) = 0$ and (ii) $f(s)> f(t)$ for all $t \in P+Q$ such that $f(t) = 0$.

Then consider the right hand side of the equation. The coefficients of $x^m$ counts size of a disjoint union of $(P;s)$-partition $f_1$ of size $k$ counted in $H_{P;s}(q)$ and $Q$-partition $f_2$ of size $n-k$ for some $k=0,\ldots,n$. Such disjoint union is equivalent to the $(P+Q;s)$-partitions counted in $H_{+Q;s}(q)$. To see this note that such $(P+Q;s)$-partition $f$ satisfies condition (i), namely  $f(s)=f_1(s)=0$. Next, we verify that $f$ satisfies the condition (ii). If $t\in P$ with $f(t)=f_1(t)=0$ then $\omega(s) > \omega(t)$ since $\sigma_1$ satisfies condition (ii). If $t\in Q$ with $f_2(t) =0$, then by assumption $\omega(s) >\omega(t)$.  Thus, we have that $H_{P+Q;s}(q) = H_{P;s}(q) \cdot G_Q(q)$ as desired.
\end{proof}

Then by applying Theorem~\ref{thm: G(P) restricted} to Lemma~\ref{lemma: GH product restricted}, we have the following corollary which we will use to prove our main result in the next section.

\begin{corollary} \label{cor: e_q identity from G product}
Let $(P+Q,\omega)$ be a labeled disjoint sum of posets and fix $s\in P$ such that $\omega(s) > \omega(t)$ for all $t\in Q$. Let $|P+Q|=n$ and $ |P|=p$. For any labeling $\omega$, we have
\begin{equation*}
   e_q^{\maj}(P+Q;s) = \qbin{n-1}{p-1}{q}e_q^{\maj}(P;s)e_q^{\maj}(Q).
\end{equation*}
\end{corollary}

\section{A Major index $q$-analogue}\label{sec: q-analogue major}

In this section we give the proof of Theorem~\ref{thm: major mobile nhlf}. The proof follows the proof of the NHLF for border strips in \cite{MPP2}. We need to first define the hook-lengths of mobile posets. Given a mobile poset $P_{\lambda/\mu}(\mathbf{p})$, define the hook-length of $(i,j) \in [\lambda]$ as following:
\begin{equation} \label{eq: def modified hook-length}
h'(i,j) = \lambda_i - i + \lambda_j' - j +1 + \sum_{a\geq i,\\ b\geq j} p_{a,b}.
\end{equation}

In other words, it is the usual hook-length of the cell in $\lambda$ plus the size of the $d$-complete posets that are attached on the segment of the border strip inside of the hook of $(i,j)$ (see Figure \ref{fig:mobile example} (a)). We provide an example of the of the theorem below.

\begin{figure}[h]
\centering
\begin{subfigure}[t]{0.3\textwidth}
\centering
\includegraphics[height = 2.5cm]{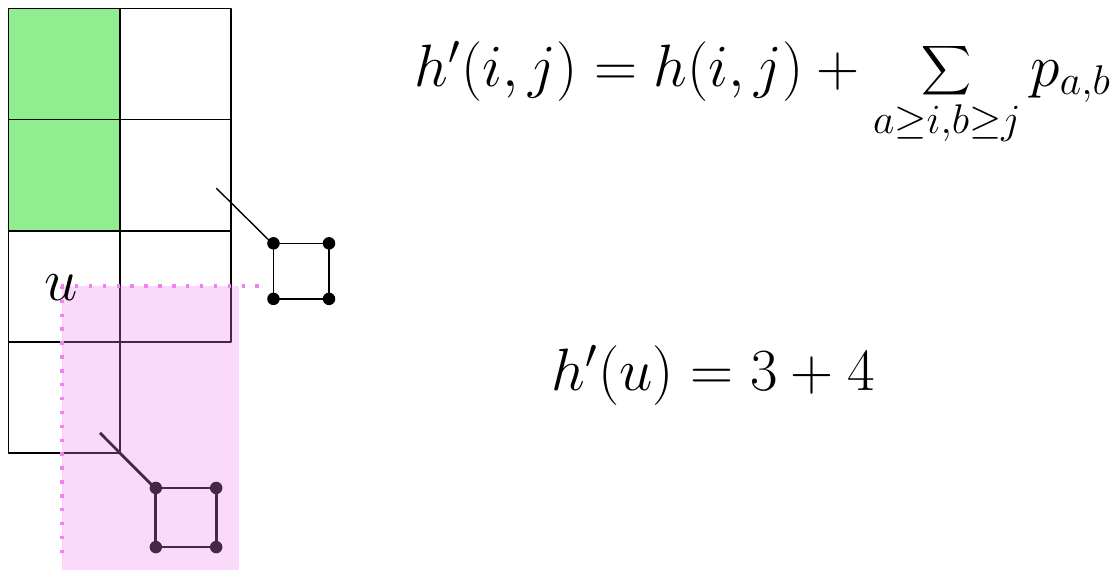}
\caption{}
\end{subfigure}
\quad
\begin{subfigure}[t]{0.3\textwidth}
\centering
\includegraphics[height = 2.5cm]{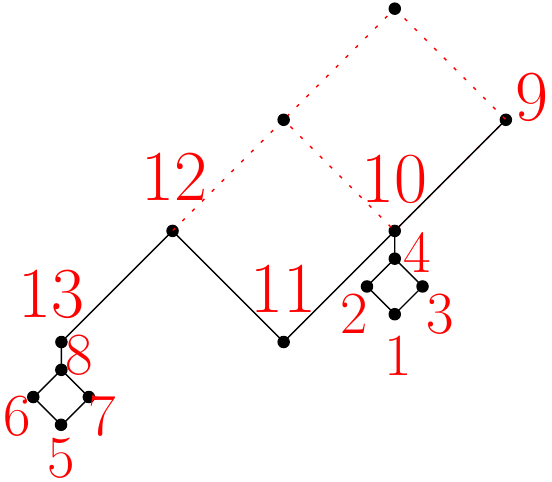}
\caption{}
\end{subfigure}
\quad
\begin{subfigure}[t]{0.3\textwidth}
\centering
\includegraphics[height=2.5cm]{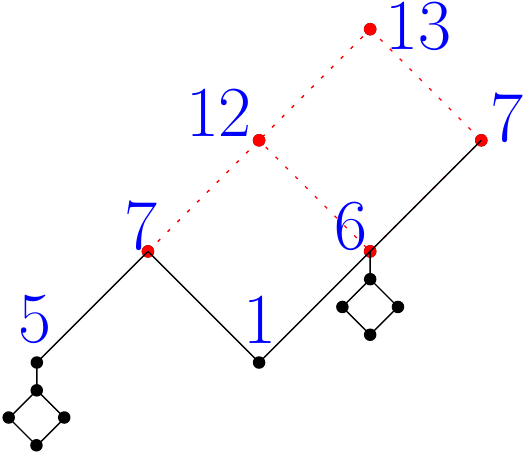}
\caption{}
\end{subfigure}
 \caption{(a): $h'(u)$ is the usual hook-length plus the size of $d$-complete posets in the shaded area. (b): mobile poset with  \textcolor{red}{reversed Schur labeling}. (c): mobile poset with  \textcolor{blue}{hook-lengths}.}
    \label{fig:mobile example}
\end{figure}

\begin{example}\label{example: major}
Consider  the  poset $(P_{2221/11},\omega)$ from Figure~\ref{fig:mobile example} (b). By Theorem~\ref{thm: major mobile nhlf}, one can check that
\begin{align*}
    e_q^{\maj}(P)  &= \frac{[13]!}{[1]^2[2]^4[3]^2}\left(\frac{q^{12}}{[5][7][1][6][7]} + \frac{q^{18}}{[5][7][12][6][7]} + \frac{q^{24}}{[5][7] [12] [13] [7]} \right)\\
    &=  q^{61} + 2q^{60} + 6q^{59}+ 11q^{58} +\cdots + 6q^{14} + 2q^{13} + q^{12}.
\end{align*}
\end{example}

We first introduce two lemmas required for the proof. We provide the proof of each lemma in  Section \ref{sec: proof of eq maj lemma} and Section \ref{sec: proof of Chevalley lemma}.

We have the following recurrence lemma for the the $q$-analogue of linear extensions for the major index.

\begin{lemma}\label{lemma: major mobile recurrence}
    For a labeled mobile poset $(P_{\lambda/\mu}(\mathbf{p}),\omega)$, where $\omega$ is a reversed Schur labeling, 
    \begin{equation}\label{eq: major mobile recurrence}
        e_q^{\maj}(P_{\lambda/\mu},\omega) = \sum_{\mu\to\nu}q^{|P_{\lambda/\nu_1}|} e_q^{\maj}(P_{\lambda/\nu},\omega_\nu)
    \end{equation}
where $P_{\lambda/\nu_1}$ is the left disconnected poset of $P_{\lambda/\nu}$, and $\omega_\nu$ is the restricted labeling of $\omega$ onto $P_{\lambda/\nu}$.
\end{lemma}

Next, we have the following Pieri--Chevalley formula. Denote the RHS of \eqref{thm: major mobile nhlf} as $H_{\lambda/\mu}(q)$. 

\[
H_{\lambda/\mu}(q) := \prod_{v\in \mathbf{p}}\frac{1}{1-q^{h(v)}}\sum_{D\in \mathcal{E}(\lambda/\mu)} q^{w'(D)} \prod_{u\in [\lambda]\setminus D} \frac{1}{1-q^{h'(u)}}
\]

\begin{lemma}\label{lemma: Chevalley qanalog maj}
\begin{equation}\label{eq: Chevalley qanalog}
    (1-q^n)\cdot H_{\lambda/\mu}(q) =   \sum_{\mu\to\nu} \frac{q^{|P_{\lambda/\nu^1}|}}{\prod_{v\in T_\nu} (1-q^{h(v)})}  \cdot H_{\lambda/\nu^1}(q)\cdot H_{\lambda/\nu^2}(q),
\end{equation}
 where $T_\nu$ is the union of the $d$-complete posets that were hanging on the removed inner corner $u$.
\end{lemma}

We are now ready to give the proof of Theorem~\ref{thm: major mobile nhlf}.

\begin{proof}[Proof of Theorem~\ref{thm: major mobile nhlf}]

We first evaluate the multivariate formula $F_{\lambda/\mu}$ at $x=q^{\lambda_i - i + 1 - \sum_{a< i} p_{a,b}}$ and $y_j = q^{j-\lambda'_j - \sum_{b\geq j} p_{a,b}}$. We denote $e_q^{\maj}(P,\omega)$ as $e_q(P)$ unless specified. We show that $e_q^{\maj}(P_{\lambda/\mu}(\mathbf{p})) = \prod_{i=1}^n(1-q^i) H_{\lambda/\mu}(q)$ by induction on $|\lambda/\mu|$ using Lemma~\ref{lemma: major mobile recurrence}. Note that $\lambda/\nu$ is disconnected and
\begin{equation}\label{eq: disconnected}
    P_{\lambda/\nu} =P_{\lambda/\nu^1}+P_{\lambda/\nu^2} +T_\nu,
\end{equation} 

where $T_\nu$ is the union of the $d$-complete posets that were hanging on the removed inner corner $u$. Denote $|P_{\lambda/\nu^j}|$ as $p_j$. By induction, we have for $j=1,2$

\[\frac{e_q(P_{\lambda/\nu^j})}{[p_j]_q!} = \prod_{i=1}^{p_j}(1-q^i)\cdot 
{H}_{\lambda/\nu^j} \cdot \frac{(1-q)^{p_j}}{\prod_{i=1}^{p_j}(1-q^i)} ={(1-q)^{p_j} } \cdot {H}_{\lambda/\nu^j},\]

and by Theorem~\ref{thm: maj d-complete poset}, for each $T_i\subset T_\nu$ we have, 
\[\frac{e_q(T_i)}{[t_i]_q!} = \frac{q^{\maj(T_i)}}{\prod_{v\in T_i}[h(v)]_q} =\frac{(1-q)^{t_i}}{\prod_{v\in T_i}(1-q^{h(v)})}.\]

 Note that $T_i$ are natural labeling, so $\maj(T_i)=0$. 
 
 Using Corollary~\ref{cor: qanalogue disjoint maj} and the equations above, we have

\[ e_q(P_{\lambda/\nu}) =\frac{\prod_{i=1}^{n-1}(1-q^i)}{\prod_{v\in T_{\nu}} (1-q^{h(v)})}  H_{\lambda/\nu^1}(q)\cdot H_{\lambda/\nu^2}(q). \]

We now apply the equation to \eqref{eq: major mobile recurrence}:
\begin{equation}\label{eq: recurrence substitued}
    e_q(P_{\lambda/\mu}) =\prod_{i=1}^{n-1}(1-q^i)\cdot  \sum_{\mu\to \nu}\frac{q^{|\lambda/\nu_1|}}{\prod_{v\in T_{\nu}} (1-q^{h(v)})} \cdot H_{\lambda/\nu^1}(q)\cdot H_{\lambda/\nu^2}(q)
\end{equation}

Note that $\lambda_1'-1 + c(u) + p_1 = |\lambda/\nu_1|$ when $\omega$ is a reversed Schur labeling. By \eqref{eq: Chevalley qanalog}, we can show the sum on the RHS of \eqref{eq: recurrence substitued} equals $(1-q^n)\cdot H_{\lambda/\mu}(q)$, completing the proof.
\end{proof}
 \begin{remark}
 We can generalize the labeling of the mobile poset by allowing non-natural labeling on the $d$-complete posets. In such case, we we would have non-trivial values for $q^{\maj(T_i)}$ in our final formula. 
 \end{remark}
 
\subsection{Proof of Lemma~\ref{lemma: major mobile recurrence}}\label{sec: proof of eq maj lemma}

To prove Lemma~\ref{lemma: major mobile recurrence}, we need the following lemmas.

Let $(P_{\lambda/\mu},\omega; U):= \{\sigma \in \mathcal{L}(P_{\lambda/\mu},\omega)\,|\, \sigma \mbox{ ends with } u \in U\}$ and 
\[
e_q^{\maj}(P_{\lambda/\mu},\omega,U):=\sum_{\sigma \in (P_{\lambda/\mu},\omega; U)} q^{\maj(\sigma)}.
\]
Unless specified otherwise, we denote this as $e_q^{\maj}(P_{\lambda/\mu},U)$.

\begin{lemma}
 For a labeled mobile poset $(P_{\lambda/\mu}(\mathbf{p}),\omega)$, where $\omega$ is a reversed Schur labeling, let $\mu\to\nu$ be the removal of an inner corner $u$, and $P_{\lambda/\nu_1}$ and $P_{\lambda/\nu_2}$ be the two disconnected parts of $\lambda/\nu$.  Then we have
 
 \begin{equation}\label{eq: linear extension partition}
    e_q^{\maj}(P_{\lambda/\mu}; \{u\}) = q^{n-1}e_q^{\maj}(P_{\lambda/\nu}; [\lambda/\nu_1]) + e_q^{\maj}(P_{\lambda/\nu}; [\lambda/\nu_2]).
\end{equation}
\end{lemma}

\begin{proof}
The linear extension on the left is $\sigma \in \mathcal{L}(P_{\lambda/\mu})$ where $\sigma = \sigma_1\dots \sigma_{n-1}\omega(u)$, we have $\sigma_1\dots\sigma_{n-1}\in \mathcal{L}(P_{\lambda/\nu})$ and $\sigma_n = u$. Then $\sigma_{n-1}$ is either an element in $Q_{\lambda/\nu_1}$ or $Q_{\lambda/\nu_2}$. If $\sigma_{n-1} \in Q_{\lambda/\nu_1}$, since $\omega$ is a reversed Schur labeling, $\omega(\sigma_{n-1}) > \omega(\sigma_n)$, so $\maj(\sigma) = \maj(\sigma_1\dots \sigma_{n-1}) + n-1$. If $\sigma_{n-1} \in Q_{\lambda/\nu_2}$, then $\omega(\sigma_{n-1})< \omega(\sigma_n)$, so $\maj(\sigma) =\maj(\sigma_1\dots\sigma_{n-1})$. 
\end{proof}

Now we are ready to prove Lemma~\ref{lemma: major mobile recurrence}.

\begin{proof}[Proof of Lemma~\ref{lemma: major mobile recurrence}]
By Corollary~\ref{cor: qanalogue disjoint maj} and a standard recurrence for $q$-binomial coefficients we have 

\begin{align}
    e_q^{\maj}(P_{\lambda/\nu}) & = \left(\qbin{n-2}{|\lambda/\nu_1|}{q} + q^{|\lambda/\nu_1|}\qbin{n-2}{|\lambda/\nu_2|-1}{q} \right) e_q^{\maj}(P_{\lambda/\nu_1}) e_q^{\maj}(P_{\lambda/\nu_2}) \nonumber\\
    &= \qbin{n-2}{|\lambda/\nu_2|}{q}e_q^{\maj}(P_{\lambda/\nu_1}) e_q^{\maj}(P_{\lambda/\nu_2}) + q^{|\lambda/\nu_1|}\qbin{n-2}{|\lambda/\nu_2|-1}{q}e_q^{\maj}(P_{\lambda/\nu_1}) e_q^{\maj}(P_{\lambda/\nu_2}).
\end{align}

Then the two parts of the sum can be interpreted in the following ways:

\begin{proposition}\label{prop: begins with nu1}
\begin{equation}\label{eq: begins with nu1}
     e_q^{\maj}(P_{\lambda/\nu}; [\lambda/\nu_1]) = \qbin{n-2}{|\lambda/\nu_1|-1,|\lambda/\nu_2|}{q}e_q^{\maj}(P_{\lambda/\nu_1}) e_q^{\maj}(P_{\lambda/\nu_2}),
\end{equation}

and

\begin{equation}\label{eq: begins with nu2}
e_q^{\maj}(P_{\lambda/\nu}; [\lambda/\nu_2]) = q^{|\lambda/\nu_1|}\qbin{n-2}{|\lambda/\nu_2|-1}{q}e_q^{\maj}(P_{\lambda/\nu_1}) e_q^{\maj}(P_{\lambda/\nu_2}).
\end{equation}
\end{proposition}

\begin{proof}
Let $\{x_1,\dots x_k\}$ be the maximal elements of $P_{\lambda/\nu_1}$. Note that
\begin{equation}\label{eq: sum of xi to nu ending with nu1}
    e_q^{\maj}(P_{\lambda/\nu};[\lambda/\nu_1]) = \sum_{i=1}^k e_q^{\maj}(P_{\lambda/\nu};x_i),
\end{equation}
and
\begin{equation}\label{eq: sum of xi to nu1}
     e_q^{\maj}(P_{\lambda/\nu_1}) = \sum_{i=1}^k e_q^{\maj}(P_{\lambda/\nu_1};x_i).
\end{equation}

We know that $P_{\lambda/\nu_1}$ and $P_{\lambda/\nu_2}$ satisfy the condition in Lemma~\ref{lemma: GH product restricted}. Then by Corollary~\ref{cor: e_q identity from G product}, we have

\begin{equation}\label{eq: emaj P;xi interpretation}
   e_q^{\maj}(P_{\lambda/\nu};x_i) = \qbin{n-2}{|\lambda/\nu_1|-1,|\lambda/\nu_2|}{q}e_q^{\maj}(P_{\lambda/\nu_1};x_i)\cdot e_q^{\maj}(P_{\lambda/\nu_2}).
\end{equation}

Applying \eqref{eq: emaj P;xi interpretation} and \eqref{eq: sum of xi to nu1} to \eqref{eq: sum of xi to nu ending with nu1}, we have the desired result.

For the second equation, we know that $e_q^{\maj}(P_{\lambda/\nu}) = e_q^{\maj}(P_{\lambda/\nu};[\lambda/\nu_1]) + e_q^{\maj}(P_{\lambda/\nu};[\lambda/\nu_2])$, so subtracting \eqref{eq: begins with nu1} from $e_q^{\maj}(P_{\lambda/\nu})$, we have the desired result as well.
\end{proof}

We now apply \eqref{eq: begins with nu1} and \eqref{eq: begins with nu2} to \eqref{eq: linear extension partition} to get the following equation. Let $p_1 = |P_{\lambda/\nu_1}|$ and $p_2 = |P_{\lambda/\nu_2}|$.

\begin{align*}
    e_q^{\maj}(P_{\lambda/\mu};\{u\}) &= q^{n-1}\qbin{n-2}{p_2}{q}e_q^{\maj}(P_{\lambda/\nu_1})\cdot  e_q^{\maj}(P_{\lambda/\nu_2}) + q^{p_1}\qbin{n-2}{p_2-1}{q}e_q^{\maj}(P_{\lambda/\nu_1})\cdot  e_q^{\maj}(P_{\lambda/\nu_2})
\end{align*}
    
After simplifying everything, we get 
\[ e_q^{\maj}(P_{\lambda/\mu};\{u\}) = q^{p_1}\cdot e_q^{\maj}(P_{\lambda/\nu}).\]

Such equation is true for all inner corners $u$ of $\mu\to\nu$, which completes the proof of Lemma~\ref{lemma: major mobile recurrence}.
\end{proof}

\subsection{Proof of Lemma~\ref{lemma: Chevalley qanalog maj}}\label{sec: proof of Chevalley lemma}

We first evaluate  $x_i=q^{\lambda_i - i + 1 - \sum_{a< i} p_{a,b}}$ and $y_j = q^{j-\lambda'_j - \sum_{b\geq j} p_{a,b}}$ in the Pieri--Chevalley formula \eqref{eq: Chevalley}. The LHS of this formula becomes
\begin{equation}\label{eq: substitute x,y in F} 
    \left.F_{\lambda/\mu} (\mathbf{x},\mathbf{y}) \right|_{\substack{x_i = q^{\lambda_i - i + 1 - \sum_{a< i} p_{a,b}},\\y_j = q^{j-\lambda'_j - \sum_{b\geq j} p_{a,b}}}} = (-1)^n \cdot \sum_{\substack{\gamma: A \to B,\\ \gamma\subset \lambda}}\prod_{(i,j)\in \gamma}\frac{q^{\lambda_j'-j + \sum_{b\geq j} p_{a,b}}}{1-q^{h'(i,j)}}
\end{equation}

By  \cite[Proposition 4.7]{MPP1} we have
\begin{align}
    \sum_{(i,j) \in [\lambda]\setminus D}\Big( (\lambda_j'-j) + \sum_{b\geq j} p_{a,b} \Big) &= \sum_{(i,j) \in [\lambda]\setminus D}\Big( (\lambda_j'-i) + \sum_{b\geq j} p_{a,b} \Big) - \sum_{(i,j)\in [\lambda]\setminus[\mu]} c(i,j), \nonumber
\end{align}
where in the last equality, we use  
 \cite[Proposition 7.16]{MPP1} to obtain,
\begin{equation}
    \label{eq: broken diagonal maj}
  \sum_{(i,j) \in [\lambda]\setminus D}\Big( (\lambda_j'-j) + \sum_{b\geq j} p_{a,b} \Big)  =  w'(\Br(D)) + \Bigg(\sum_{(i,j) \in [\lambda]\setminus D} \sum_{b\geq j }p_{a,b} - \sum_{\substack{a\geq i, b\geq j\\(i,j)\in \Br(D)}}p_{a,b}\Bigg) - \sum_{(i,j)\in [\lambda]\setminus[\mu]} c(i,j)
\end{equation}
where $w'(\Br(D)) = \sum_{(i,j)\in \Br(D)}h'(i,j)$, and the subtraction in the second sum is the $d$-complete posets that are included in the new hook-lengths of the broken diagonals. We denote the quantity in parenthesis on the RHS as $\widehat{p}_D$.
\[\widehat{p}_D := \sum_{(i,j) \in [\lambda]\setminus D} \sum_{b\geq j }p_{a,b} - \sum_{\substack{a\geq i, b\geq j\\(i,j)\in \Br(D)}}p_{a,b}\]
We claim that $\widehat{p}_D$ is invariant among $D$.

\begin{lemma}
The quantity $\widehat{p}_D$ is invariant among all $D\in \mathcal{E}(\lambda/\mu)$.
\end{lemma}
\begin{proof}
 We prove by induction on excited moves $\beta: D \to D'$. Denote $a_i = (s_i,t_i)$ as the active cell and $b = (s +1, t+1),b' = (s+1, t)$ be the old and new broken diagonal of an excited move $\beta:D \to D'$, Then for each excited move $\beta$, $[\lambda]\setminus D' = ([\lambda]\setminus D) \setminus \{b\} \cup \{a\}$. Also, $\Br(D') = \Br(D) \setminus \{b\} \cup \{b'\}$. Thus

\begin{figure}[h]
    \centering
    \includegraphics[height=1cm]{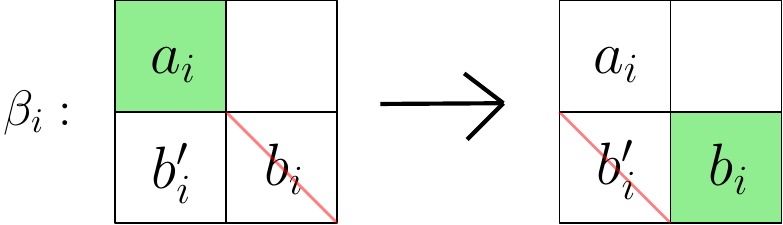}
    \caption{For each excited move $\beta_i$, we place the old broken diagonal $b_i$ with the new broken diagonal $b_i'$.}
    \label{fig: Excited move 2}
\end{figure}

\begin{align*}
    \widehat{p}_{D'} &= \left(\sum_{(i,j)\in[\lambda]\setminus D} \sum_{b\geq j} p_{a,b} - \sum_{b \geq t_i+1} p_{a,b} + \sum_{b \geq t_i} p_{a,b}\right) - \left( \sum_{\substack{a\geq i, b\geq j\\ (i,j)\in \Br(D)}} p_{a,b} - \sum_{\substack{a\geq s_i+1,\\ b\geq t_i+1}}p_{a,b} + \sum_{\substack{a\geq s_i+1, \\b\geq t_i}}p_{a,b} \right)\\
     &= \left(\sum_{(i,j)\in[\lambda]\setminus D} \sum_{b\geq j} p_{a,b} + \sum_{b = t_i} p_{a,b} \right) - \Bigg( \sum_{\substack{a\geq i, b\geq j\\ (i,j)\in \Br(D)}} p_{a,b} + \sum_{\substack{a\geq s_i+1,\\b = t_i}} p_{a,b}\Bigg)\\
     &= \widehat{p}_D + \sum_{b=t_i} p_{a,b} - \sum_{\substack{a\geq s_i+1,\\b = t_i}} p_{a,b}
\end{align*}

For $\lambda/\mu$ border strip, there cannot be any $d$-complete posets hanging above an active cell in the same column, so the difference on the RHS of the equation above is zero and so $\widehat{P}_{D'}=\widehat{P}_D$. 
\end{proof}

We then let $\widehat{p}_{D} = \widehat{p}$ for all $D\in \mathcal{E}(\lambda/\mu)$. Putting $\widehat{p}$ and $c(i,j)$ outside of the sum, we can rewrite \eqref{eq: substitute x,y in F} as:

\begin{equation} \label{eq: F to H}
      \left.F_{\lambda/\mu} (\mathbf{x}|\mathbf{y}) \right|_{\substack{x_i = q^{\lambda_i - i + 1 - \sum_{a< i} p_{a,b}},\\ y_j = q^{j-\lambda'_j - \sum_{b\geq j} p_{a,b}}}} = (-1)^n \cdot q^{\widehat{p}_{\mu}- \con(\lambda/\mu)} \prod_{v\in \mathbf{p}}(1-q^{h(v)})\cdot H_{\lambda/\mu},
\end{equation}
where $\con(\lambda/\mu):=\sum_{(i,j)\in [\lambda/\mu]}c(i,j)$. Next we see what happens to the RHS of Pieri--Chevalley formula \eqref{eq: Chevalley} when we evaluate at $x_i = q^{\lambda_i - i + 1 - \sum_{a< i} p_{a,b}}$ and $ y_j = q^{j-\lambda'_j - \sum_{b\geq j} p_{a,b}}$, where the sizes $p_{a,b}$ appearing in the sums are for the $d$-complete posets in $P_{\lambda/\mu}(\mathbf{p})$.  The linear factor on the RHS of Pieri--Chevalley formula becomes
\begin{equation} \label{eq: evaluation of x1 - y1}
\left.\frac{1}{x_1-y_1}\right|_{\substack{x_i = q^{\lambda_i - i + 1 - \sum_{a< i} p_{a,b}},\\ y_j = q^{j-\lambda'_j - \sum_{b\geq j} p_{a,b}}}} \,=\, \frac{(-1)^{n-1} }{q^{\lambda_1 - \sum_{a< 1}p_{a,b}}-q^{1-\lambda'_1 - \sum_{b\geq1}p_{a,b}}} =\frac{(-1)^{n}q^{\lambda_1'-1+\sum_{(a,b)\in [\lambda/\mu]}p_{a,b}}}{(1-q^n)}.
\end{equation}

Given an inner corner removed $\mu \to \nu$, denote $\mathbf{p}_1$ and $\mathbf{p}_2$ be the set of $d$-complete poset hanging on $P_{\lambda/\nu^1}$ and $P_{\lambda/\nu^2}$ respectively. Then for the shapes $\lambda/\nu^k$ where $k=1,2$  we have

\begin{equation} \label{eq: F to H nu1 nu2}
      \left.F_{\lambda/\nu^k} (\mathbf{x}|\mathbf{y}) \right|_{\substack{x_i = q^{\lambda_i - i + 1 - \sum_{a< i} p_{a,b}},\\ y_j = q^{j-\lambda'_j - \sum_{b\geq j} p_{a,b}}}} = (-1)^{|\lambda/\nu^k|} \cdot q^{\widehat{p}_{\nu^k}- \con(\lambda/\nu^k)} \prod_{v\in \mathbf{p}_k}(1-q^{h(v)})\cdot H_{\lambda/\nu^k},
\end{equation}

\[\frac{\prod_{v\in \mathbf{p}_1}(1-q^{h(v)}) \cdot \prod_{v\in \mathbf{p}_2}(1-q^{h(v)})}{\prod_{v\in \mathbf{p}}(1-q^{h(v)})} = \frac{1}{\prod_{v\in T_\nu} (1-q^{h(v)})}\]

Thus, by  \eqref{eq: F to H}, \eqref{eq: evaluation of x1 - y1}, and \eqref{eq: F to H nu1 nu2} the Pieri--Chevalley formula \eqref{eq: Chevalley} evaluated at such $x_i$ and $y_j$ becomes,
\begin{equation}
    q^{\widehat{p}_{\mu}- \con(\lambda/\mu)} H_{\lambda/\mu} =\frac{q^{\lambda_1'-1+\sum_{(a,b)\in [\lambda/\mu]}p_{a,b}}}{(1-q^n)}\sum_{\mu\to\nu} \frac{q^{\widehat{p}_{\nu_1} + \widehat{p}_{\nu_2} -  \con(\lambda/\nu_1) - \con(\lambda/\nu_2)}}{\prod_{v\in T_\nu} (1-q^{h(v)})}  H_{\lambda/\nu_1} H_{\lambda/\nu_2}.
\end{equation}
Note that for each inner corner $u:\mu\to \nu$, we have  $\con(\lambda/\mu)-\con(\lambda/\nu_1)-\con(\lambda/\nu_2) = c(u)$. Thus the previous equation becomes

\begin{equation*}\label{eq: substitute xi, yi calcuation}
    (1-q^n)H_{\lambda/\mu} =\sum_{\mu\to\nu} \frac{q^{\lambda'_1-1+ c(u)+\sum_{(a,b)\in [\lambda/\mu]}p_{a,b}+ \widehat{p}_{\nu_1}+ \widehat{p}_{\nu_2}-\widehat{p}_\mu}}{\prod_{v\in T_\nu} (1-q^{h(v)})} H_{\lambda/\nu_1}H_{\lambda/\nu_2}.
\end{equation*}

\begin{lemma} 
For the equation defined above,
\[
\sum_{(a,b)\in [\lambda/\mu]}p_{a,b} + \widehat{p}_{\nu_1} + \widehat{p}_{\nu_2} - \widehat{p}_{\mu} = |\mathbf{p}_1|.
\]

\end{lemma}

\begin{proof}
Consider $\sum_{b\geq 1}p_{a,b} + \widehat{p}_{\nu_1} + \widehat{p}_{\nu_2} - \widehat{p}_{\mu}$. We know that $\widehat{p}$ is invariant among the excited diagrams, so without loss of generality, assume $D=[\mu]$.

We have $\Br([\nu_1]) \cup \Br([\nu_2]) \cup \{u_0\}= \Br([\lambda/\mu]),$ where $u_0 = (u_1+1,u_2)$ is a broken diagonal of $[\mu]$ below $u = (u_1,u_2)$. Denote $p_i$ as the size of the $d$-complete posets hanging on $\lambda/\nu_i$. Then, 

\begin{figure}[h]
    \centering
    \includegraphics[height = 3cm]{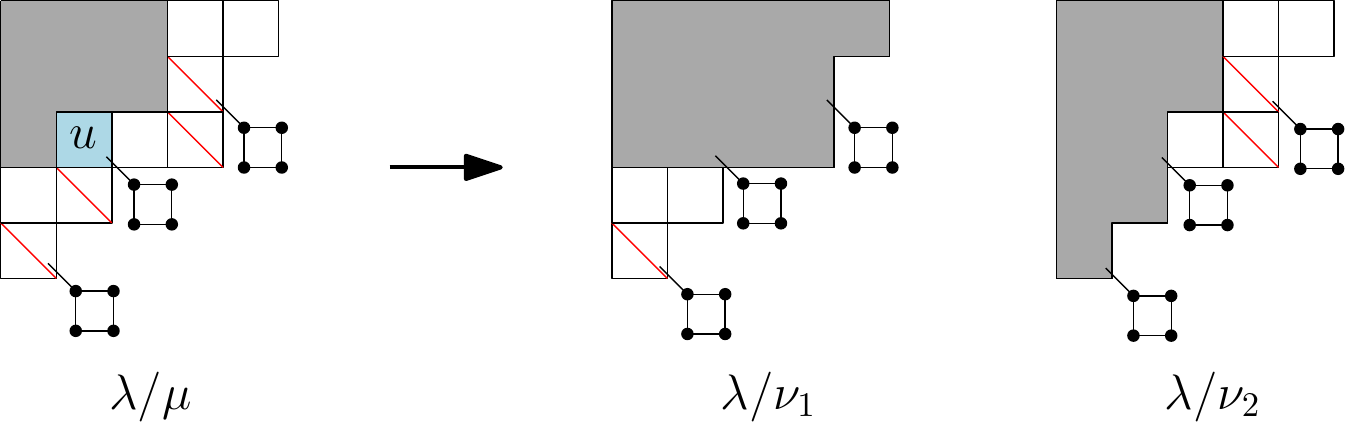}
    \caption{$[\lambda/\mu] = [\lambda/\nu_1] + [\lambda/\nu_2] + u$. Note that in $\lambda/\nu_1$, we are missing a broken diagonal underneath $u$.}
    \label{fig:mu to nu1 and nu2}
\end{figure}

\begin{multline}
\label{eq: calculation of p hat}
    \sum_{b\geq 1}p_{a,b} + \widehat{p}_{\nu_1} + \widehat{p}_{\nu_2} - \widehat{p}_{\mu} = \sum_{b\geq 1}p_{a,b} + \Bigg(\sum_{(i,j) \in [\lambda/\nu_1]} \sum_{b\geq j }p_{a,b} + \sum_{(i,j) \in [\lambda/\nu_2]} \sum_{b\geq j }p_{a,b} - \sum_{(i,j) \in [\lambda/\mu]} \sum_{b\geq j }p_{a,b}\Bigg)\\
    \hspace{10pt} -\Bigg( \sum_{\substack{a\geq i, b\geq j\\(i,j)\in \Br([\nu_1])}}p_{a,b} + \sum_{\substack{a\geq i, b\geq j\\(i,j)\in \Br([\nu_2])}}p_{a,b} - \sum_{\substack{a\geq i, b\geq j\\(i,j)\in \Br([\mu])}}p_{a,b}\Bigg)
\end{multline}
Let $S_1$ and $S_2$ be the sums in parenthesis on the RHS above. Since $\lambda/\nu_1$ and $\lambda/\nu_2$ do not contain the inner corner $u$ (see Figure~\ref{fig:mu to nu1 and nu2}) then $S_1$ and $S_2$ simplify to 
\begin{equation}\label{eq: p hat calculation}
  S_1 = \sum_{b\geq u_2} p_{a,b}, \qquad S_2 = \sum_{a\geq u_1+1, b\geq u_2} p_{a,b}.  
\end{equation}

Thus Equation~\eqref{eq: calculation of p hat} becomes
\begin{align*}
 \sum_{b\geq 1}p_{a,b} + \widehat{p}_{\nu_1} + \widehat{p}_{\nu_2} - \widehat{p}_{\mu}   &= \sum_{b\geq 1}p_{a,b} - \Bigg( \sum_{b\geq u_2} p_{a,b} - \sum_{a\geq u_1+1, b\geq u_2} p_{a,b}\Bigg).
    \label{eq: widehat p calculation}
\end{align*}
Note that the equation in the parenthesis counts the number of $d$-complete posets hanging on and to the the right of $u$. These are exactly the $d$-complete posets on $P_{\lambda/\nu^1}$.
\end{proof}

Finally, note that $\lambda_1'-1 + c(u) + |\mathbf{p}_1| = |P_{\lambda/\nu_1}|$. Then we can simplify \eqref{eq: substitute xi, yi calcuation} to obtain the desired formula.

\section{Application: bounds for the number of linear extensions}\label{sec: application}

In this section, we provide a short proof of Corollary~\ref{thm: mobile nhlf} along with an example. We also discuss an application of the formula to the bounds of generalized Euler numbers.

\subsection{The case of $q=1$}

\begin{proof}[proof of Corollary~\ref{thm: mobile nhlf}]
Taking $q=1$ for Theorem~\ref{thm: major mobile nhlf}, we get the proof of Corollary~\ref{thm: mobile nhlf}. One can also prove Corollary~\ref{thm: mobile nhlf} directly by evaluating the multivariate formula $F_{\lambda/\mu}$ at $x_i = \lambda_i -i+1 - \sum_{a<i} p_{a,b}$ and $y_j = j-\lambda_j' -\sum_{b\geq j}p_{a,b}$
\end{proof}

 We give the example of the theorem below.

\begin{example}
Consider the mobile poset $P_{2221/11}$ from Figure~\ref{fig:mobile example} (b). Then by Corollary~\ref{thm: mobile nhlf} we have
\begin{align}
    e(P) = \frac{13!}{2^\cdot 4^2}\left(\frac{1}{5\cdot 6\cdot 7^2} + \frac{1}{5\cdot 6 \cdot 7^2 \cdot 12} + \frac{1}{5 \cdot 7^2 \cdot 12 \cdot 13} \right) = 33000.
\end{align}

\end{example}

\begin{figure}[t]
\centering
\begin{subfigure}[normal]{0.4\textwidth}
\centering
\includegraphics[height=2.5cm]{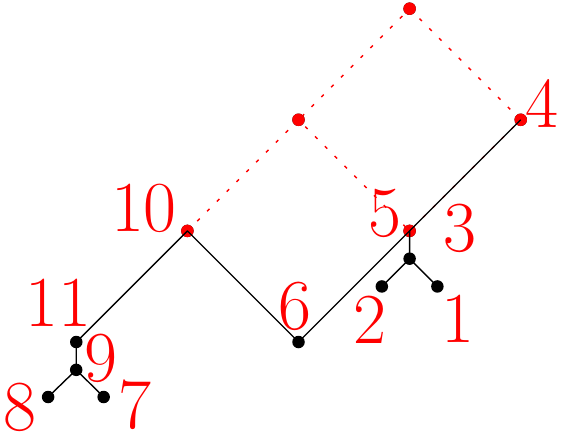}
\caption{}
\end{subfigure}
\quad
\begin{subfigure}[normal]{0.4\textwidth}
\centering
\includegraphics[height=2.5cm]{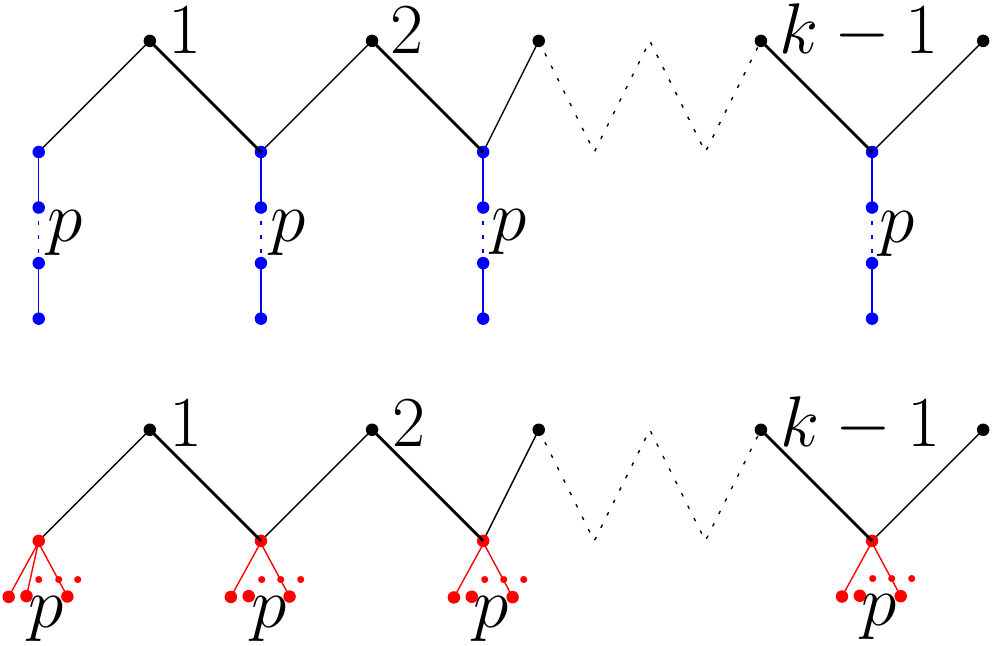}
\caption{}
\end{subfigure}

    \caption{(a) the \textcolor{red}{$\omega_\inv$ labeled} mobile tree poset, (b) illustration of the posets $\mathcal{C}_p(k)$ (top) and $\mathcal{A}_p(k)$ (bottom).}
    \label{fig:mobile example2}
\end{figure}

\subsection{Bounds to generalizations of Euler numbers}

As an application to Corollary~\ref{thm: mobile nhlf} gives bounds to $e(P_{\lambda/\mu}(\mathbf{p}))$ just as in \cite{MPPasymptotics}.

\begin{corollary}\label{cor: boundary}
For any mobile poset $e(P_{\lambda/\mu}(\mathbf{p}))$ of size $n$, 
\[\frac{n!}{H(\mathbf{p})\prod_{u\in [\lambda/\mu]}h'(u)} \leq e(P_{\lambda/\mu}(\mathbf{p})) \leq |\E(\lambda/\mu)|\cdot \frac{n!}{H(\mathbf{p})\prod_{u\in [\lambda/\mu]}h'(u)}\]
where $[\lambda/\mu]$ is the border strip of the mobile poset.
\end{corollary}

\begin{proof}
For any skew shape $\lambda/\mu$, we have $[\mu] \in \mathcal{E}(\lambda/\mu)$, so the lower bound is given by Corollary~\ref{thm: mobile nhlf}. For the upper bound, note that under the excited move, the product $\prod_{v\in [\lambda/\mu]}h'(u)$ increases. Then this product is minimal when the excited diagram is $[\mu]$ and the upper bound follows from Corollary~\ref{thm: mobile nhlf}.
\end{proof}

For more detail about asymptotic of linear extensions of skew shaped tableaux, see \cite{MPPasymptotics}.

One application of the formula is that it provides bounds to generalizations of \emph{Euler numbers} defined in \cite{GMM}. The authors give two generalizations of Euler number using two different families of posets, up-down posets with $k-1$ downs and chains (or anti-chains) of size $p$  hanging on every minimal element,denoted as $\mathcal{C}_p(k)$ and $\mathcal{A}_p(k)$ (see Figure~\ref{fig:mobile example} (b)). See \cite[\href{https://oeis.org/A332471}{A332471}]{OEIS} and \cite[\href{https://oeis.org/A332568}{A332568}]{OEIS} for examples of these sequences.

\begin{corollary}\label{cor: zigzag bound}

\[\frac{(2k+kp)!}{ (p+1)!^k(2p+3)^{k-1}(p+2)} \leq e(\mathcal{C}_p(k)) \leq 
\Cat(k)\cdot \frac{(2k+ kp)!}{ (p+1)!^k(2p+3)^{k-1}(p+2)}\]

\[\frac{(2k+ kp)!}{(p+1)^k(2p+3)^{k-1}(p+2)} \leq e(\mathcal{A}_p(k)) \leq 
\Cat(k)\cdot \frac{(2k + kp)!}{(p+1)^k(2p+3)^{k-1}(p+2)}\]
where $\mathcal{Z}$ is the up-down border strip with $k-1$ many down steps. 
\end{corollary}

\begin{proof}
The result follows from Corollary~\ref{cor: boundary}, a routine calculations of hooks, and the fact that the excited diagrams of up-down posets are given by the Catalan numbers \cite[Corollary 8.1]{MPP2}

\end{proof}

\section{An Inversion index $q$-analogue}\label{sec: q-analogue inversion}
In this section we give an example and the proof of Theorem~\ref{thm: inv q-analogue}. Unless specified otherwise, $(P_{\lambda/\mu}(\mathbf{p}),\omega)$ is a labeled mobile tree poset. 

\subsection{Labeling of the poset for the case of inversion index}

The mobile tree poset must satisfy a very specific labeling for the case of the inversion statistic. One of the reasons why is because of the condition stated in Proposition~\ref{prop: qanalogue disjoint poset inv}. Another reason is so that the labeling needs to satisfy Lemma~\ref{lem: inv recurrence}. To satisfy both conditions, we must label the poset in the following way: let $\{u_1,\dots, u_k\}$ be the list of inner corners of $P_{\lambda/\mu}$ from $(1,\lambda_1)$ to $(\lambda'_1,1)$, and $u_0 = (1,\lambda_1)$ and $u_{k+1} = (\lambda_1',1)$. Partition the mobile posets into $P_1,\dots,P_{k+1}$ such that for each $P_i$, it contains all elements $(s,t) \in P_{\lambda/\mu}$ such that $u_{i}< t \leq u_{i-1}$ and all the elements of rooted trees hanging on such $(s,t)$. If $u_{k+1} = (\lambda_1',1)$, then $P_{k+1} = u_{k+1}$. Starting from $P_1$, we label each $P_i$ such that all the hanging rooted trees are naturally labeled and the elements in the border strip have reversed Schur labeling. See Figure~\ref{fig:inversion labeling} for an example. We denote such labeling as $\omega_{\inv}$.
\begin{figure}[h]
    \centering
    \includegraphics[height= 3cm]{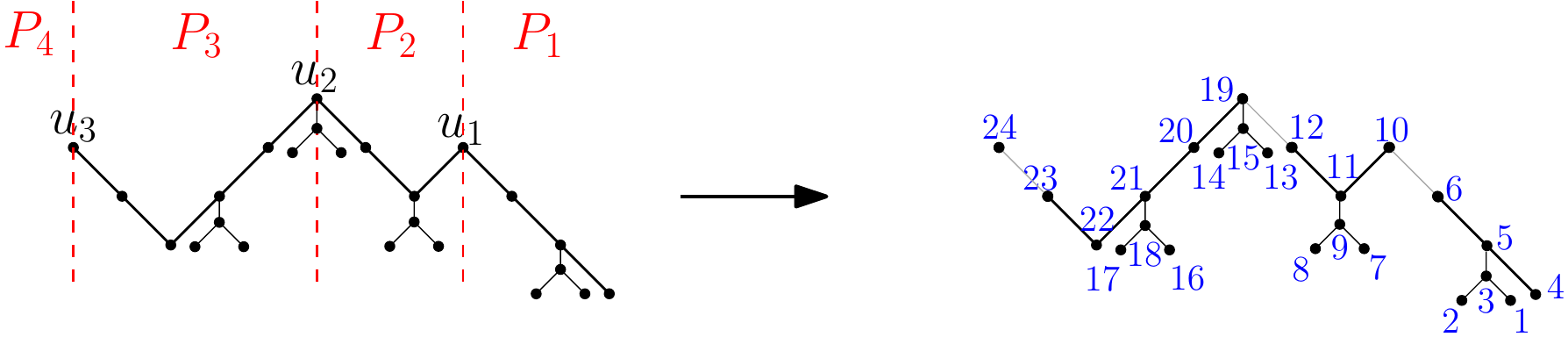}
    \caption{A mobile tree poset with labeling $\omega_{\inv}$. On the left we show the partitions of the poset into $P_1,\dots, P_4$. We label each $P_i$ so that the $d$-complete posets are naturally labeled and the elements in the border-strip have reversed Schur labeling.}
    \label{fig:inversion labeling}
\end{figure}

\begin{example}\label{example: inversion}
Consider labeled the mobile poset $(P_{2221/11},\omega_{\inv})$ from Figure~\ref{fig:mobile example2} (a). Then by Theorem~\ref{thm: inv nhlf} we have

\begin{align*}
    e_q^{\inv}(P) &= \frac{[11]!}{[1]^4[3]^2}\left(\frac{q^4}{[4][6][1][5][6]} + \frac{q^9}{[4][6][10][5][6]} + \frac{q^{14}}{[4][6] [10] [11] [6]} \right)\\
    & = q^{38} + 4 q^{37} + 9 q^{36} + 17 q^{35} +\cdots + 9 q^6 + 4q^5 + q^4.
\end{align*}
\end{example}

\begin{remark}
Note that Theorem~\ref{thm: inv nhlf} is only for mobile trees, where the $d$-complete posets are restricted to rooted trees. This is because there is no known hook-length formula for $e_q^\inv(P)$ when $P$ is a general $d$-complete poset.
\end{remark}

We need the following recursion for the inversion index $q$-analogue.

\begin{lemma}\label{lem: inv recurrence}
\begin{equation}
    e_q^{\inv}(P_{\lambda/\mu},\omega) = \sum_{\mu\to\nu} q^{n-\omega(u)} e_q^{\inv}(P_{\lambda/\nu}, \omega_\nu),
\end{equation}
where $\omega$ is a reversed Schur labeling and $\omega(u)$ is the label of the inner corner $u$ from $\mu\to\nu$.
\end{lemma}

 We also need the following Pieri--Chevalley formula for the inversion index. Denote the the RHS of \eqref{thm: inv q-analogue} as $\widetilde{H}_{\lambda/\mu}(q)$:
 
 \[\widetilde{H}_{\lambda/\mu}(q):= \prod_{v\in \mathbf{p}}\frac{1}{1-q^{h(v)}}\sum_{D\in \mathcal{E}(\lambda/\mu)} q^{w(D)+p_D} \prod_{u\in [\lambda]\setminus D} \frac{1}{1-q^{h'(u)}}, \]

\begin{lemma}\label{lem: Chevalley q analog inv lem}

\begin{equation}\label{eq: Chevalley q analogue inv}
    (1-q^n)\cdot \widetilde{H}_{\lambda/\mu} =   \sum_{\mu\to\nu} \frac{q^{q^{n-\omega_{\inv}(u)}}}{\prod_{v\in T_\nu} (1-q^{h(v)})}  \cdot \widetilde{H}_{\lambda/\nu^1}\cdot \widetilde{H}_{\lambda/\nu^2},
\end{equation}
where $\omega_{\inv}(u)$ is the label of $u$, the inner corner from $\mu\to\nu$
\end{lemma}

We provide the proof of the lemmas in Section~\ref{subsec: inv Chevalley proof}.

\begin{proof}
Recall that for a fixed tree mobile $P_{\lambda/\mu}$, a linear extension of $\sigma \in \mathcal{L}(P_{\lambda/\mu})$ consists of an  inner corner of $\lambda/\mu$ followed by a linear extension of the remaining poset of shape $\lambda/\nu$, where $\mu \to \nu$. Conversely, given a linear extension $\sigma' \in \mathcal{L}(P_{\lambda/\nu})$, by inserting the new element in the beginning we obtain a linear extension of $P_{\lambda/\mu}$. Note that $\inv(\sigma) = \inv(\sigma') + n - \omega$, where $n-\omega$ is the number of inversion caused by the inner corner. The result follows from this correspondence
\end{proof}

We are now ready to give the proof of Theorem~\ref{thm: inv nhlf}.

\begin{proof}[Proof of Theorem \ref{thm: inv q-analogue}]

Similarly as in the case of major index, we show that $e_q^{\inv}(P_{\lambda/\mu}(\mathbf{p})) ={\prod_{i=1}^n(1-q^i)}\cdot \widetilde{H}_{\lambda/\mu}(q)$ by induction on $|\lambda/\mu|$ using Lemma~\ref{lem: inv recurrence}. Recall $P_{\lambda/\nu}$ can be expressed as \eqref{eq: disconnected}. By induction and Theorem~\ref{thm: maj d-complete poset}, we have

\[\frac{e_q(P_{\lambda/\nu^j})}{[p_j]_q!} = \prod_{i=1}^{p_j}(1-q^i)\cdot \widetilde{H}_{\lambda/\nu^j} \cdot \frac{(1-q)^{p_j}}{\prod_{i=1}^{p_j}(1-q^i)} ={(1-q)^{p_j} } \cdot  \widetilde{H}_{\lambda/\nu^j},\] 

where $p_j = |P_{\lambda/\nu^j}|$. Also, for each $T_i \subset T_{\nu}$, 
\begin{equation}\label{eq: inv calculation}
    \frac{e_q^{\inv}(T_i)}{[t_i]_q!} = \frac{q^{\inv(T_i)}}{\prod_{v\in T_i}[h(v)]_q} =\frac{(1-q)^{t_i}}{\prod_{v\in T_i}(1-q^{h(v)})}.
\end{equation}

 Note that $T_i$ are naturally labeled, so $\inv(T_i)=0$. Using Proposition~\ref{prop: qanalogue disjoint poset inv} and \eqref{eq: inv calculation}, we have

\[ e_q(P_{\lambda/\nu}) =\frac{\prod_{i=1}^{n-1}(1-q^i)}{\prod_{v\in T_{\nu}} (1-q^{h(v)})}  \widetilde{H}_{\lambda/\nu^1}\cdot \widetilde{H}_{\lambda/\nu^2}. \]

By this equation and Lemma~\ref{lem: inv recurrence}. 
\begin{equation}\label{eq: recurrence substituted inv}
    e_q(P_{\lambda/\mu}) ={\prod_{i=1}^{n-1}(1-q^i)}\sum_{\mu\to \nu} \frac{ q^{n-\omega_{\inv}(u)}}{\prod_{v\in T_{\nu}} (1-q^{h(v)})}\cdot  \widetilde{H}_{\lambda/\nu^1}\cdot \widetilde{H}_{\lambda/\nu^2}.
\end{equation}

By \eqref{eq: Chevalley q analogue inv}, we can show the sum on the RHS of \eqref{eq: recurrence substituted inv} equals $(1-q^n)\cdot \widetilde{H}_{\lambda/\mu}$, completing the proof.
\end{proof}

\subsection{Proof of Lemma~\ref{lem: Chevalley q analog inv lem}}\label{subsec: inv Chevalley proof}

We first evaluate $F_{\lambda/\mu}(\mathbf{x},\mathbf{y})$ at $x_i=q^{\lambda_i - i + 1 - \sum_{a< i} p_{a,b}}$ and $y_j = q^{j-\lambda'_j - \sum_{b\geq j} p_{a,b}}$.
\begin{equation}\label{eq: substitute x,y F inv} 
    \left.F_{\lambda/\mu} (\mathbf{x},\mathbf{y}) \right.\mid_{\substack{x_i = q^{\lambda_i - i + 1 - \sum_{a< i} p_{a,b}},\\ y_j = q^{j-\lambda'_j - \sum_{b\geq j} p_{a,b}}}} = (-1)^n \sum_{\substack{\gamma: A \to B,\\ \gamma\subset \lambda}}\prod_{(i,j)\in \gamma}\frac{q^{\lambda_j'-j + \sum_{b\geq j} p_{a,b}}}{1-q^{h'(i,j)}}
\end{equation}

By  \cite[Prop 4.7]{MPP1} and \cite[Lemma 7.17]{MPP1}, we have
\begin{align}
    \sum_{(i,j) \in \lambda\setminus D}\Big( (\lambda_j'-j) + \sum_{b\geq j} p_{a,b} \Big) &= \sum_{(i,j) \in [\lambda]\setminus D}\Big( (\lambda_j'-i) + \sum_{b\geq j} p_{a,b} \Big) - \sum_{(i,j)\in [\lambda]\setminus[\mu]} c(i,j) \nonumber \\ 
    &=  w(\Br(D)) + \sum_{(i,j) \in [\lambda]\setminus D} \sum_{b\geq j }p_{a,b} - \sum_{(i,j)\in [\lambda]\setminus[\mu]} c(i,j) \label{eq: broken diagonal inv},
\end{align}

where $c(i,j) = j-i$ and $w(\Br(D)) = \sum_{(i,j)\in Br(D)} h(i,j)$. Note that unlike the case of major index, we do not include the size of the rooted trees in to $w(\Br(D))$ (see \eqref{eq: broken diagonal maj}) .

Denote $\Tilde{p}_{\lambda/\mu}:=\sum_{(i,j) \in [\lambda]/[\mu]} \sum_{b\geq j }p_{a,b}$. For each $D\in \mathcal{E}(\lambda/\mu)$, we have

\[\sum_{(i,j)\in[\lambda]\setminus D}\sum_{b\geq j}p_{a,b} - \Tilde{p}_{\lambda/\mu} = \sum_{(i,j)\in [\mu]/D}\sum_{b=j}p_{a,b} = p_D.\]

Then $\Tilde{p}_{\lambda/\mu}$ and $c(i,j)$ from \eqref{eq: broken diagonal inv} do not depend on $D$, so we can take them outside of the sum to rewrite \eqref{eq: substitute x,y F inv} as:
\begin{equation}\label{eq: xi yi substituted inversion}
      \left.F_{\lambda/\mu} (\mathbf{x}|\mathbf{y}) \right.\mid_{\substack{x_i = q^{\lambda'_1 - i + 1 - \sum_{a< i} p_{a,b}},\\ y_j = q^{j-\lambda'_j - \sum_{b\geq j} p_{a,b}}}} = (-1)^n \cdot q^{\Tilde{p}_{\lambda/\mu}-\con(\lambda/\mu)}\prod_{v\in \mathbf{p}}(1-q^{h(v)})\cdot \widetilde{H}_{\lambda/\mu}
\end{equation}

Now as done for the case of major index in Section~\ref{sec: q-analogue major}, we evaluate the Pieri--Chevalley formula at such $x_i$ and $y_j$. Then applying \eqref{eq: xi yi substituted inversion} to \eqref{eq: Chevalley}, and simplifying everything as we did in the case of major index, we have

\begin{equation}\label{eq: recurrence in H inv}
    (1-q^n)\widetilde{H}_{\lambda/\mu} = \sum_{\mu\to\nu} \frac{q^{\lambda_1-1+c(u)+ \sum_{b\geq 1}p_{a,b}+\Tilde{p}_{\lambda/\nu_1}+ \Tilde{p}_{\lambda/\nu_2} - \Tilde{p}_{\lambda/\mu} }}{\prod_{v\in T_{\nu}} (1-q^{h(v)})}\widetilde{H}_{\lambda/\nu_1}\widetilde{H}_{\lambda/\nu_2}.
\end{equation}

Note that from \eqref{eq: p hat calculation}, this is equivalent to

\[\sum_{b\geq 1}p_{a,b} +\Tilde{p}_{\lambda/\nu_1} + \Tilde{p}_{\lambda/\nu_2} - \Tilde{p}_{\lambda/\mu} = \sum_{b\geq 1}p_{a,b} - \sum_{b\geq u_2}p_{a,b},\]

where $u_2$ is the column of the the inner corner $u=(u_1,u_2)$. It is left to show the following lemma to complete the proof.

\begin{lemma}
Let $(P_{\lambda/\mu}(\mathbf{p}),\omega_{\inv})$ be a mobile tree poset of size $n$ with a labeling $\omega_{\inv}$ and $u$ be the inner corner for $\mu\to\nu$. Then we have,
\[n-\omega_{\inv}(u) = \lambda'_1 -1 + c(u) +  \sum_{b\geq 1}p_{a,b} - \sum_{b\geq u_2}p_{a,b}.\]
\end{lemma}

\begin{proof}
First, we show that for a border strip $(P_{\lambda/\mu},\omega)$ of size $n_0$ with a reversed Schur labeling we have $n_0-\omega(x) = \lambda'_1 -1 + c(x)$ for all $x\in P_{\lambda/\mu}$. Note that in a border strip, there is only one element per content. Also, for Schur labeling of a border strip, we have $\omega(\lambda_1',1) = 1$ and $\omega(1,\lambda_1) = n_0$. The element $(\lambda_1',1)$ satisfies the equation. Then as you follow the border strip, the content decreases by one while the label increases by one, so the rest of the elements satisfy the equation $n_0-\omega(x) = \lambda'_1 -1 + c(x)$.

Now for a labeled mobile tree poset $(P_{\lambda/\mu}(\mathbf{p}),\omega_{\inv})$ of size $n$, for any $x\in P_{\lambda/\mu}(\mathbf{p})$, $\omega_{\inv}(x)$ gets shifted by $\sum_{b\geq x_2}p_{a,b}$. Then $\omega_{\inv}(x) = \omega(x) + \sum_{b\geq x_2}p_{a,b}$. Also, we have that $n = n_0 + \sum_{b\geq 1}p_{a,b}$. Then applying such shifts to the equation obtained from a border strip, we have the desired equation.
\end{proof}

Then we can simplify \eqref{eq: recurrence in H inv}, completing the proof of Lemma~\ref{lem: Chevalley q analog inv lem}.

\section{Final remarks}\label{sec: final remarks}

\subsection{Theorem~\ref{thm: mobile nhlf} for border-strips}

In \cite{MPP1} Morales, Pak, and Panova gave a proof of Theorem~\ref{thm: MPP1 maj} using factorial Schur functions. In \cite{MPP2} the same authors gave another proof of Theorem~\ref{thm: MPP1 maj} reducing it to the case of border strips. The latter proof included an analogue of Lemma~\ref{lemma: Chevalley qanalog maj} to border strips, but there was no explicit analogue of Lemma~\ref{lemma: major mobile recurrence}. Instead they relied on an identity \cite[Lemma 7.2]{MPP2} proved using factorial Schur functions. Our Lemma~\ref{lemma: major mobile recurrence} can be reduced to the case of border strips as follows. 

\begin{corollary}\label{cor: border strip recursion}
For a labeled border-strip poset $(Q_{\lambda/\mu},\omega)$, where $\omega$ is a reversed Schur labeling,
\[e_q^{\maj}(Q_{\lambda/\mu},\omega) = \sum_{\mu\to\nu}q^{|Q_{\lambda/\nu_1}|} e_q^{\maj}(Q_{\lambda/\nu},\omega_\nu),\]
where $Q_{\lambda/\nu_1}$ is the left disconnected poset of $Q_{\lambda/\nu}$, and $\omega_\nu$ is the restricted labeling of $\omega$ onto $Q_{\lambda/\nu}$.
\end{corollary}

Then using the Pieri--Chevalley formula (\eqref{eq: Chevalley} proved in \cite{MPP2}) and Corollary~\ref{cor: border strip recursion}, we obtain a proof of Theorem~\ref{thm: MPP1 maj} for border strips without using factorial Schur functions.

\subsection{Bijective proof between maj and inv index for border strips}

The inversion statistic analogue of Lemma~\ref{lemma: major mobile recurrence} is Lemma~\ref{lem: inv recurrence}. In the case of border strips $Q_{\lambda/\mu}$, since $n-\omega(u) = |Q_{\lambda/\nu_1}|$, we obtain the $\emph{same}$ recurrence as $e_q^{\maj}(Q_{\lambda/\mu})$ as in Corollary~\ref{cor: border strip recursion}. Thus, we obtain the following equation of the $q$-analogues for border strips. 

\begin{corollary}\label{cor: inv = maj}
For a border strip $Q_{\lambda/\mu}$,
\[e_q^{\inv}(Q_{\lambda/\mu},\omega) = e_q^{\maj}(Q_{\lambda/\mu},\omega),\]
where $\omega$ is a reversed Schur labeling
\end{corollary}

This identity can also be proved bijectively using  Foata's classical bijection on permutations, denoted by $\varphi$, defined as follows (see \cite[Sec. 1.4]{Sta1}). Let $w = w_1\cdots w_n \in  \mathfrak{S}_n$, and we define $\gamma_1,\dots \gamma_n$, where $\gamma_k$ is a permutation of $\{w_1,\dots,w_k\}$. Let $\gamma_1 = w_1$. For each $k\geq 1$, if the last letter of $\gamma_k$ is greater (respectively smaller) than $w_{k+1}$, then split $\gamma_k$ after each letter greater (respectively smaller) than $w_{k+1}$. To obtain $\gamma_{k+1}$, cyclically shift each compartment of $\gamma_k$ to the right, then place $w_{k+1}$ at the end. We set $\varphi(w) = \gamma_n$. We have the following theorem.

\begin{theorem}[Foata \cite{F}]
Let $\varphi$ be the Foata bijection. For all $w \in \mathfrak{S}_n$, \[\des(w^{-1}) = \des(\varphi(w)^{-1}).\]
\end{theorem}

Because Foata's bijection preserves descent sets, we have the following bijection between the major and inversion index.

\begin{lemma}
Given a $\sigma \in \mathcal{L}(P, \omega)$, where $\omega$ is a (reversed) Schur labeling, $\varphi(\sigma)$ is also in $\mathcal{L}(P,\omega)$, and
$\maj(\sigma) = \inv(\varphi(\sigma))$
\end{lemma}

More detailed information about the equidistribution of major and inversion index in trees can be found in \cite{BW}. The situation for mobile posets is more subtle since the $q$-analogues do not agree (see Example~\ref{example: major} and Example~\ref{example: inversion}).

\subsection{Mobiles of general skew shapes}

The formula \eqref{eq:NHLF} holds true for all posets coming from skew shapes, but the combinatorial proof of the formula is restricted to the case of border strips. Recall that a mobile is obtained by hanging $d$-complete posets from a border-strip. It would be interesting to see if Theorem~\ref{thm: major mobile nhlf} holds for posets where the border strip is replaced by general skew shape. Calculations suggest that the Naruse formula \eqref{thm: major mobile nhlf} would need some adjustments.

On the other hand, we use a version of Pieri--Chevalley formula and the recurrence for our proof. There is a version of the Pieri--Chevalley formula  for general skew shapes, shown algebraically by Ikeda and Naruse \cite{IkedaNaruse}, and combinatorically by Konvalinka \cite{Konvalinka}.

\subsection{Relation with Naruse-Okada hook-length formula}\label{subsec: skew d-complete}

Naruse-Okada \cite{NaruseOkada} have a different $q$-analogue of $e_q^{\maj}(P,\omega)$ for a family called \emph{skew d-complete} posets, which intersects with the family of mobile posets \cite[Section 6.1]{GMM}. 

\begin{definition}\cite{NaruseOkada}
A \emph{skew $d$-complete poset} is a $d$-complete poset $P$ with an order filter $I$ removed. We denote such a poset by $P\setminus I$.
\end{definition}

The Naruse-Okada formula for counting linear extensions of skew $d$-complete posets uses the hook-length of \emph{excited peaks} (see \cite[Section 6]{MPP1} and \cite{NaruseOkada})instead of broken diagonals.

\begin{theorem}[{Naruse-Okada \cite{naruse2014}}]\label{thm: NaruseOkada skew d-complete}
Let $P\setminus I$ be a skew $d$-complete poset with $n$ elements. Then
\[e_q^{\maj}(P\setminus I) = \prod_{i=1}^n(1-q^i) \sum_{D\in \mathcal{E}(P\setminus I)}\frac{\prod_{v\in B(D)}q^{h(v)}}{ \prod_{v\in P\setminus D}(1-q^{h(v)})},\]
where $h(v)$ is the hook length of element $v$ in $P\setminus I$ and $B(D)$ is a set of excited peaks of $D$.
\end{theorem}

For posets that are both mobiles and skew $d$-complete, the notion of hook-lengths are the same (see Figure~\ref{fig:skew d-complete labeling}). Then for such posets Theorem~\ref{thm: NaruseOkada skew d-complete} at $q=1$ and Corollary~\ref{thm: mobile nhlf} agree.

However, the $q$-analogues in Theorem~\ref{thm: mobile nhlf} and Theorem~\ref{thm: NaruseOkada skew d-complete} are different (see Example~\ref{ex: skew d-complete}). This is because the NHLF formula for skew $d$-complete posets uses the natural labeling of the poset as opposed to the reversed Schur labeling. For the case of skew shapes, their $q$-analogue agrees with the reverse plane partition $q$-analogue of the Naruse formula in (see \cite{MPP1} Corollary 6.17) instead of SSYT $q$-analogue, which uses the Schur labeling (see Figure~\ref{fig:skew d-complete labeling}).

\begin{figure}[h]
\centering
\begin{subfigure}[normal]{0.3\textwidth}
\centering
\includegraphics[height=1.5cm]{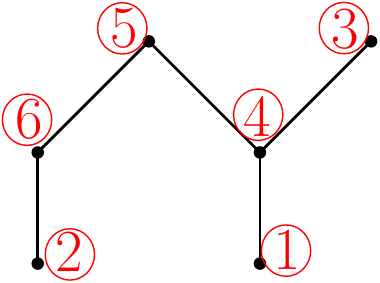}
\caption{}
\end{subfigure}
\quad
\begin{subfigure}[normal]{0.3\textwidth}
\centering
\includegraphics[height=1.5cm]{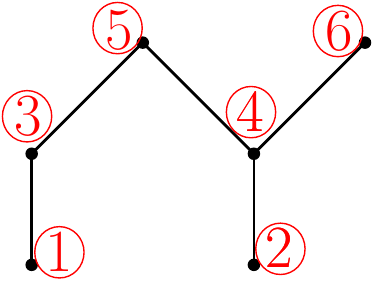}
\caption{}
\end{subfigure}
\begin{subfigure}[normal]{0.3\textwidth}
\centering
\includegraphics[height=1.5cm]{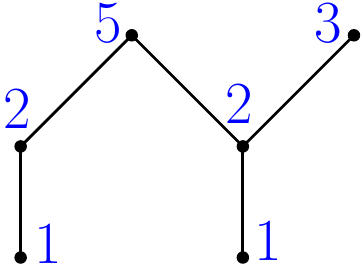}
\caption{}
\end{subfigure}

    \caption{A skew $d$-complete poset $P\setminus I$ that is also a mobile labeled using (a) the Schur labeling $\omega'$ and (b) the natural labeling $\omega$. (c) The hook lengths of the elements in the poset.}
    \label{fig:skew d-complete labeling}
\end{figure}

\begin{example}\label{ex: skew d-complete}
Consider the poset $Q=P\setminus I$ in Figure~\ref{fig:skew d-complete labeling} that is both a mobile poset and a skew $d$-complete poset \cite[Ex. 6.3]{GMM}. If we label it using the reversed Schur labeling on the border strip and natural labeling on the $d$-complete posets, then by Theorem~\ref{thm: major mobile nhlf}, we have
\begin{align*}
    e_q^{\maj}(Q,\omega') &= q^{11} + 2q^{10} + 3q^9 + 3q^8 + 3q^7 + 2q^6 + 1q^5 + q^4\\
    &=[6]!\left(\frac{q^4}{[1][1][2][2][3][5]} + \frac{q^7}{[1][1][2][3][5][6]}\right).
\end{align*}

Now, label the same skew $d$-complete poset $Q=P\setminus I$ using the natural labeling. Then, by Naruse--Okada formula (Theorem~\ref{thm: NaruseOkada skew d-complete}), we have
\begin{align*}
    e_q^{\maj}(Q,\omega) &= q^9 + q^8 + 2q^7 + 2q^6 + 2q^5 + 2q^4 + 2q^3 + 2q^3 + 2q^2 + q + 1\\
    &=[6]!\left(\frac{q^0}{[1][1][2][2][3][5]} + \frac{q^6}{[1][1][2][3][5][6]}\right).
\end{align*}

\end{example}

It would be interesting to see if one can give a proof of Theorem~\ref{thm: NaruseOkada skew d-complete} using the technique from \cite{MPP2}. This would involve proving a variation of Lemma~\ref{lemma: major mobile recurrence} where $\omega$ is a natural labeling instead of reversed Schur labeling.

\section*{Acknowledgement}
I would like to thank Alejandro Morales for introducing this problem and for all of his guidance throughout this project. I would also like to thank Stefan Grosser, Jacob Matherne, and Soichi Okada for helpful comments.

\bibliographystyle{plain}
\bibliography{Naruse}

\begin{thebibliography}{10}

\bibitem{BW}
A.~Bj\"{o}rner and M.~L. Wachs.
\newblock {$q$}-hook length formulas for forests.
\newblock {\em J. Combin. Theory Ser. A}, 52(2):165--187, 1989.

\bibitem{Brightwell-Winkler}
G.~Brightwell and P.~Winkler.
\newblock Counting linear extensions.
\newblock {\em Order}, 8(3):225--242, 1991.

\bibitem{F}
D.~Foata.
\newblock On the {N}etto inversion number of a sequence.
\newblock {\em Proc. Amer. Math. Soc.}, 19:236--240, 1968.

\bibitem{FRT}
J.~S. Frame, G.~de~B. Robinson, and R.~M. Thrall.
\newblock The hook graphs of the symmetric groups.
\newblock {\em Canad. J. Math.}, 6:316--324, 1954.

\bibitem{GMM}
A.~Garver, S.~Grosser, J.~P. Matherne, and A.~H. Morales.
\newblock Counting linear extensions of posets with determinants of hook
  lengths.
\newblock {\em SIAM J. Discrete Math.}, 35(1):205--233, 2021.

\bibitem{IkedaNaruse}
T.~Ikeda and H.~Naruse.
\newblock Excited {Y}oung diagrams and equivariant {S}chubert calculus.
\newblock {\em Trans. Amer. Math. Soc.}, 361(10):5193--5221, 2009.

\bibitem{Konvalinka}
M.~Konvalinka.
\newblock A bijective proof of the hook-length formula for skew shapes.
\newblock {\em European J. Combin.}, 88:103104, 14, 2020.

\bibitem{MPP2}
A.~H. Morales, I.~Pak, and G.~Panova.
\newblock Hook formulas for skew shapes ii. combinatorial proofs and
  enumerative applications.
\newblock {\em SIAM Journal on Discrete Mathematics}, 31(3):1953--1989, 2017.

\bibitem{MPPasymptotics}
A.~H. Morales, I.~Pak, and G.~Panova.
\newblock Asymptotics of the number of standard {Y}oung tableaux of skew shape.
\newblock {\em European J. Combin.}, 70:26--49, 2018.

\bibitem{MPP1}
A.~H. Morales, I.~Pak, and G.~Panova.
\newblock Hook formulas for skew shapes {I}. {$q$}-analogues and bijections.
\newblock {\em J. Combin. Theory Ser. A}, 154:350--405, 2018.

\bibitem{naruse2014}
H.~Naruse.
\newblock Schubert calculus and hook formula.
\newblock Slides at 73rd S\'{e}minaire Lotharingien de Combinatoire, Strobl,
  \url{http://www.emis.de/journals/SLC/wpapers/s73vortrag/naruse.pdf}, 2014.

\bibitem{NaruseOkada}
H.~Naruse and S.~Okada.
\newblock Skew hook formula for {$d$}-complete posets via equivariant
  {$K$}-theory.
\newblock {\em Algebr. Comb.}, 2(4):541--571, 2019.

\bibitem{Pro}
R.~A. Proctor.
\newblock d-complete posets generalize young diagrams for the hook product
  formula: Partial presentation of proof.
\newblock {\em RIMS}, 1913:120--140, 2014.

\bibitem{OEIS}
N.~J.~A. Sloane.
\newblock The {O}n-{L}ine {E}ncyclopedia of {I}nteger {S}equences.
\newblock published electronically at \url{https://oeis.org}.

\bibitem{Sta3}
R.~P. Stanley.
\newblock Theory and application of plane partitions. {I}, {II}.
\newblock {\em Studies in Appl. Math.}, 50:167--188; ibid. 50 (1971), 259--279,
  1971.

\bibitem{Sta1}
R.~P. Stanley.
\newblock {\em Enumerative combinatorics. {V}olume 1}, volume~49 of {\em
  Cambridge Studies in Advanced Mathematics}.
\newblock Cambridge University Press, Cambridge, second edition, 2012.

\end{thebibliography}

\end{document}